\documentclass[a4paper,10pt]{amsart}

\usepackage{mathtools,amssymb,amsfonts,stmaryrd,amsmath,amsthm,enumerate,hyperref,color}
\usepackage[utf8]{inputenc}
\usepackage{enumitem}

\usepackage[all,cmtip]{xy}     
\usepackage{wrapfig,tikz,tikz-cd}
\usetikzlibrary{arrows,decorations.pathreplacing}
\usepackage{caption}

\usepackage[capitalise]{cleveref}

\usepackage[parfill]{parskip}
\usepackage[margin=1.2in]{geometry}
\AtBeginDocument{\addtocontents{toc}{\protect\setlength{\parskip}{0pt}}}

\newtheorem{theorem}{Theorem}[section]
\newtheorem{lemma}[theorem]{Lemma}
\newtheorem{proposition}[theorem]{Proposition}
\newtheorem{corollary}[theorem]{Corollary}

\newtheorem{theorema}{Theorem}

\theoremstyle{definition}
\newtheorem{definition}[theorem]{Definition}
\newtheorem{remark}[theorem]{Remark}
\newtheorem{example}[theorem]{Example}
\newtheorem{question}{Question}

\theoremstyle{remark}

\begin{document}

\title[Cohomology operations for moment-angle complexes]{Cohomology operations for moment-angle complexes\\ and resolutions of Stanley--Reisner rings} 

\date{24 May 2023}

\author[S.~Amelotte]{Steven Amelotte}
\address{Department of Mathematics, Western University, London, ON N6A 5B7, Canada}
\email{samelot@uwo.ca} 

\author[B.~Briggs]{Benjamin Briggs}
\address{Department of Mathematical Sciences,
University of Copenhagen,
Universitetsparken 5
DK-2100 Copenhagen \O}
\email{bpb@math.ku.dk}


\thanks{For part of this work, Amelotte was hosted by the Institute for Computational and Experimental Research in Mathematics in Providence, RI, supported by the National Science Foundation under Grant No.\ 1929284.
Briggs was funded by the European Union under the Grant Agreement no.\ 101064551 (Hochschild).}

\begin{abstract}
A fundamental result in toric topology identifies the cohomology ring of the moment-angle complex $\mathcal{Z}_K$ associated to a simplicial complex $K$ with the Koszul homology of the Stanley--Reisner ring of $K$. By studying cohomology operations induced by the standard torus action on the moment-angle complex, we extend this to a topological interpretation of the minimal free resolution of the Stanley--Reisner ring. The exterior algebra module structure in cohomology induced by the torus action recovers the linear part of the minimal free resolution, and we show that higher cohomology operations induced by the action (in the sense of Goresky--Kottwitz--MacPherson) can be assembled into an explicit differential on the resolution. Describing these operations in terms of Hochster's formula, we recover and extend a result due to Katth\"an. We then apply all of this to study the equivariant formality of torus actions on moment-angle complexes. For these spaces, we obtain complete algebraic and combinatorial characterisations of which subtori of the naturally acting torus act equivariantly formally.
\end{abstract}

\maketitle

\tableofcontents

\section{Introduction}

Equivariantly formal torus actions are among the richest and best understood examples of group actions studied in geometry and topology. The action of a torus $T$ on a space $X$ is called \emph{equivariantly formal} if the equivariant cohomology $H^*_T(X)$ is free as a module over the polynomial ring $H^*(BT)$. This condition makes the relationship between ordinary and equivariant cohomology as simple as possible, and often allows $H^*_T(X)$ to be described in terms of fixed point data (as in the Borel localization theorem and its variants) or low-dimensional orbits of the action (as in the Chang--Skjelbred lemma \cite{CS} and its generalisations to the Atiyah--Bredon sequence \cite{FP}). Examples of equivariantly formal torus actions include Hamiltonian actions on compact symplectic manifolds, GKM manifolds and all $T$-spaces with cohomology concentrated in even degrees.

An important source of torus actions is the \emph{moment-angle complex}, a central construction in toric topology that functorially assigns a space $\mathcal{Z}_K$ with a $T^m=(S^1)^m$-action to each simplicial complex $K$ on~$m$ vertices. These $T^m$-spaces play a universal role in toric topology. For example, every quasitoric manifold (including every smooth projective toric variety) is diffeomorphic to the quotient $\mathcal{Z}_K/T^{m-n}$ of a moment-angle complex by some restriction of its standard $T^m$-action to a maximal freely acting subtorus $T^{m-n}\subseteq T^m$ (cf.\ \cite[7.3]{BPbook}). On the other hand, the homotopy quotient of $\mathcal{Z}_K$ by the entire $T^m$-action yields the \emph{Davis--Januszkiewicz space} $DJ_K$, whose cohomology ring $H^*(DJ_K;k)=H^*_{T^m}(\mathcal{Z}_K;k)$ is isomorphic to the Stanley--Reisner ring
\[
k[K]=S/(v_{i_1}\!\cdots v_{i_q} \,:\, \{i_1,\ldots,i_q\} \notin K),
\]
where $S=H^*(BT^m;k)=k[v_1,\ldots,v_m]$ is a polynomial ring with generators in degree~$2$.

The (homotopy) quotients of moment-angle complexes by other subgroups $H\subseteq T^m$ have recently been investigated by many authors~\cite{FRANZ,FU,LI,LS,LMM,PZ}. One motivation for this paper is to answer the following question:

\begin{question} \label{question}
Let $K$ be a simplicial complex on the vertex set $[m]=\{1,\ldots,m\}$. For which subtori $H\subseteq T^m$ is the $H$-action on the moment-angle complex $\mathcal{Z}_K$ equivariantly formal? 
\end{question}

Rather than directly computing the equivariant cohomology $H^*_H(\mathcal{Z}_K)$ as a module over a polynomial ring for every subtorus $H\subseteq T^m$, we approach \cref{question} from a Koszul dual perspective by studying the ordinary cohomology $H^*(\mathcal{Z}_K)$ as a module over an exterior algebra.

Indeed, the $T^m$-action on $\mathcal{Z}_K$ equips $H^*(\mathcal{Z}_K)$ with a natural (in $K$) structure of a graded module over $\Lambda=\Lambda(\iota_1,\ldots,
\iota_m)$, with generators $\iota_j$ acting by derivations of degree $-1$. In \cref{s_exterior_module} we begin a detailed study of this module structure. As a graded algebra, the ordinary cohomology of $\mathcal{Z}_K$ is well understood in terms of the homological algebra of the Stanley--Reisner ring (see~\cite{BP}):
\[
H^*(\mathcal{Z}_K;k) \cong \operatorname{Tor}^S(k[K],k).
\]
We describe the $\Lambda$-module structure on both sides of this isomorphism, lifting it to a cochain-level isomorphism of differential graded $\Lambda$-modules. 

Hochster's formula for the Betti numbers of a Stanley--Reisner ring yields another description of the cohomology of a moment-angle complex, namely a decomposition $H^*(\mathcal{Z}_K) \cong \bigoplus_{J\subseteq [m]} \widetilde{H}^*(K_J)$ in terms of the reduced simplicial cohomology of full subcomplexes $K_J\subseteq K$. We show that this can also be upgraded to an isomorphism of $\Lambda$-modules, where the action of each $\iota_j$ on the components of Hochster's decomposition coincides, up to sign, with the maps
\begin{equation}\label{eq_Katthan_maps}
\widetilde{H}^{*}(K_J) \longrightarrow \widetilde{H}^{*}(K_{J\smallsetminus j})
\end{equation}
induced by the inclusions of full subcomplexes $K_{J\smallsetminus j} \hookrightarrow K_J$ (see \cref{comm_diagrams}). Together with \cref{intro_th_A} below, this recovers a combinatorial description due to Katth\"an \cite{K} of the linear part of the minimal free resolution of $k[K]$.

The derivations $\iota_j\colon H^*(\mathcal{Z}_K)\to H^{*-1}(\mathcal{Z}_K)$ extend to a family of higher operations
\[
\delta_s\colon H^*(\mathcal{Z}_K)\longrightarrow H^{*-2\deg(s)+1}(\mathcal{Z}_K),
\]
indexed by the monomials of $S$. In the context of equivariant cohomology, the higher cohomology operations induced by a torus action were introduced by Goresky, Kottwitz and MacPherson~\cite{GKM} as obstructions to equivariant formality. As with other higher cohomology operations,~$\delta_s$ is generally only defined on the kernels of lower degree operations, taking values with indeterminacy depending on these previous operations. For torus actions on smooth manifolds, this indeterminacy can be avoided by identifying de Rham cohomology with the space of harmonic forms, as in~\cite{AP} and~\cite{CJ}. In our case, since a moment-angle complex is not always a manifold and we wish to work with coefficients in an arbitrary field~$k$, we make use of the homological perturbation lemma in \cref{s_HB_pert} to obtain an explicit family of higher operations which are well-defined endomorphisms of $H^*(\mathcal{Z}_K)$. The relevance to equivariant formality is made clear by the main results of \cref{s_HB_pert}, which show that these operations assemble into a differential yielding the minimal Hirsch--Brown model for the action of any subtorus $H\subseteq T^m$ on $\mathcal{Z}_K$. In particular, for $H=T^m$ we obtain the following.

\begin{theorema}\label{intro_th_A}
If $K$ is a simplicial complex on the vertex set $[m]$, then
\[
\big(S \otimes H^*(\mathcal{Z}_K;k), \delta\big), \qquad \delta=\sum_{J\subseteq [m]}v_J\otimes\delta_J
\]
is the minimal free resolution of the Stanley--Reisner ring $k[K]$, or equivalently the minimal Hirsch--Brown model for the action of $T^m$ on $\mathcal{Z}_K$. Here $v_J=v_{i_1}\!\cdots v_{i_q}$ and $\delta_J$ is the cohomology operation indexed by $v_J$, for each $J=\{i_1,\ldots,i_q\}\subseteq [m]$. 
\end{theorema}

Describing the linear part of the resolution of $k[K]$ amounts to understanding the primary operations $\iota_j=\delta_{\{j\}}$, and these are given in combinatorial terms by Katth\"an's formula \eqref{eq_Katthan_maps}. In \cite{K}, Katth\"an poses the question of describing the quadratic part of the resolution in similar terms; in other words, how can the secondary operations $\delta_{\{ij\}}$ be understood in terms of Hochster's decomposition $H^*(\mathcal{Z}_K) \cong \bigoplus_{J\subseteq [m]} \widetilde{H}^*(K_J)$? More generally, one can ask

\begin{question} \label{question2}
How can the higher operations $\delta_s$ on  $H^*(\mathcal{Z}_K;k)$---or equivalently, the differentials in the minimal free resolution of $k[K]$---be described in terms of the combinatorics of $K$?
\end{question}

In \cref{comb_models_sec} we give combinatorial models for the higher operations by identifying them with differentials in a Mayer--Vietoris spectral sequence defined entirely in terms of simplicial cochains on full subcomplexes of $K$. From the perspective of commutative algebra, this yields a description of the minimal $S$-free resolution of the Stanley--Reisner ring, purely in terms of Hochster's decomposition and the combinatorics of $K$, up to the indeterminacy in their definition, which remains an obstacle to a complete description. This recovers and extends Katth\"an's theorem.

Special attention is paid to the secondary operations (yielding the quadratic part of the resolution of $k[K]$) in \cref{s_MV_secondary}. We find that, just as the primary operations $\delta_{\{i\}}$ are determined by the maps $K_{J\smallsetminus i} \hookrightarrow K_J$ for all $J\subseteq [m]$, the secondary operations $\delta_{\{ij\}}$ are essentially determined by the natural inclusions
\[
K_{J\smallsetminus i}\cup K_{J\smallsetminus j} \hookrightarrow K_J \quad\text{ and }\quad K_{J\smallsetminus i}\cup K_{J\smallsetminus j} \hookrightarrow \Sigma K_{J\smallsetminus ij}.
\]
More generally, the behaviour of the higher operations is governed by inclusions of \emph{face deletions} $K_J\smallsetminus F=\bigcup_{i\in F}K_{J\smallsetminus i}$ for $F\in K_J$; see \cref{s_equivariant_formality_general}.

Returning to the question of equivariant formality, we find that this condition can be read off from the minimal free resolution of the Stanley--Reisner ring. We show in \cref{s_J_closed} that for any subtorus $H\subseteq T^m$, the $H$-action on $\mathcal{Z}_K$ is equivariantly formal if and only if $k[K]$ is $\mathcal{J}$-closed in the sense of Diethorn~\cite{DIETHORN}, where $\mathcal{J}\subseteq S$ is a certain ideal generated by linear polynomials associated to the torus $H$.

To obtain combinatorial characterisations of equivariant formality, we first reduce \cref{question} to the case of coordinate subtorus actions on $\mathcal{Z}_K$, that is,  actions of subtori of the form
\[
T^I=\big\{(t_1,\ldots,t_m)\in T^m \,:\, t_i=1\text{ for }i\notin I\big\}
\]
for some $I\subseteq [m]$; see \cref{s_reduction}. The simplest case to consider is that of a coordinate circle $S^1_j=T^{\{j\}}\subseteq T^m$. In this case, equivariant formality is determined by the $\Lambda$-module structure on $H^*(\mathcal{Z}_K)$ alone: we show in \cref{1-torus_ef} that the $S^1_j$-action on $\mathcal{Z}_K$ is equivariantly formal if and only if the derivation $\iota_j\colon H^*(\mathcal{Z}_K)\to H^{*-1}(\mathcal{Z}_K)$ is trivial.  By the combinatorial description \eqref{eq_Katthan_maps} of the primary operations, it is equivalent that the inclusion of full subcomplexes $K_{J\smallsetminus j} \hookrightarrow K_J$ induces the trivial map in reduced simplicial cohomology for all $J\subseteq [m]$ with $j\in J$.

For the $T^I$-actions on $\mathcal{Z}_K$ with $|I|>1$, equivariant formality is not determined by the $\Lambda$-module $H^*(\mathcal{Z}_K)$, as the vanishing of higher operations is also necessary. Although the higher operations are not simply induced by inclusions among the full subcomplexes appearing in Hochster's formula, it nonetheless turns out that the vanishing of all $\delta_J$ with $J\subseteq I$ is  detected by inclusions of the (not necessarily full) face deletion subcomplexes.

\begin{theorema}
Let $K$ be a simplicial complex on the vertex set $[m]$ and let $I\subseteq [m]$. Then the following conditions are equivalent:
\begin{enumerate}[label={\normalfont(\alph*)}]
\item the coordinate $T^I$-action on $\mathcal{Z}_K$ is equivariantly formal over $k$;
\item the cohomology operations $\delta_J$ vanish on $H^*(\mathcal{Z}_K;k)$ for all $J\subseteq I$;
\item the Stanley--Reisner ring $k[K]$ is $\mathcal{J}_I$-closed, where $\mathcal{J}_I=(v_i \,:\, i\notin I)$;
\item $K_J\smallsetminus(I\cap J) \hookrightarrow K_J$ induces the trivial map on $\widetilde{H}^\ast(\; ;k)$ for all $J\subseteq [m]$.
\end{enumerate}
\end{theorema}

In case $K$ is flag (or equivalently, the Stanley--Reisner ring $k[K]$ is quadratic), condition (d) above can be simplified considerably. In this situation, we obtain the following characterisation of equivariant formality purely in terms of the combinatorics of $K$.

\begin{theorema}
Let $K$ be a flag complex on $[m]$ and let $I\subseteq [m]$. The coordinate $T^I$-action on $\mathcal{Z}_K$ is equivariantly formal if and only if $I\in K$ and $K_{\{i,j\}}\ast K_{I\smallsetminus \{i,j\}} \subseteq K$ for every missing edge $\{i,j\} \notin K$.
\end{theorema}

An interesting consequence is that the equivariant formality of these actions is independent of the field $k$ in the flag case. We give examples in \cref{s_dependence_on_char} illustrating that this is not true in general.

\subsection*{Acknowledgements} The authors are very grateful to Rachel Diethorn and Matthias Franz for many helpful comments on a draft of this work. The first author would also like to thank Graham Denham for a helpful conversation regarding {\tt Macaulay2}.

\section{Preliminaries and notation}

Throughout this paper we fix a natural number $m$ and a simplicial complex $K$ on the vertex set $[m]=\{1,\ldots,m\}$. The \emph{moment-angle complex} over $K$ is the subspace of $(D^2)^m$ defined by the polyhedral product
\begin{equation} \label{eq_MAC_def}
\mathcal{Z}_K = \bigcup_{\sigma\in K} (D^2,S^1)^\sigma, \qquad (D^2,S^1)^\sigma=\left\{ (z_1,\ldots,z_m)\in (D^2)^m \,:\, z_i\in S^1\text{ if }i\notin \sigma\right\}.
\end{equation}
Note that the coordinatewise action of the torus $T^m=(S^1)^m$ on $(D^2)^m$ restricts to an action of $T^m$ on $\mathcal{Z}_K$. 

We also fix a field $k$, and we write
\[
S=k[v_1,\ldots,v_m] \quad \text{and} \quad \Lambda=\Lambda(\iota_1,\ldots,\iota_m)
\]
for the polynomial algebra and exterior algebra over $k$, generated by variables in bijection with~$[m]$. We think of the exterior variables $\iota_i$ as dual to the polynomial variables $v_i$. Both $S$ and $\Lambda$ are multigraded by $\mathbb{N}^m$. In particular, when $J=\{j_1,\ldots,j_r\}\subseteq [m]$, we write $v_J=v_{j_1}\!\cdots v_{j_r}$ for the monomial with (squarefree) multidegree $J$.

The \emph{Stanley--Reisner ring} of $K$ is the multigraded algebra
\[
k[K]=S/(v_{i_1}\!\cdots v_{i_q} \,:\, \{i_1,\ldots,i_q\} \notin K).
\]
The Koszul complex of $k[K]$ is the multigraded dg algebra
\[
\big(k[K] \otimes \Lambda(u_1,\ldots,u_m),d\big)
\]
with each $u_i$ given homological degree $1$ and multidegree $(0,\ldots,1,\ldots,0)$ (a $1$ in the $i$th position), and with differential determined by $d(u_i)=v_i$ and the graded Leibniz rule. As before, for a subset $J=\{j_1,\ldots,j_r\}\subseteq [m]$ with $j_1<\cdots<j_r$, we will use the notation $u_J=u_{j_1}\wedge\cdots\wedge u_{j_r}$.

For each $i$ there is a unique derivation $\iota_i$ on the Koszul complex with (homological) degree $-1$, determined by 
\begin{equation}\label{e_lambda_action}
\iota_i(u_i)=1,\ \iota_i(u_j)=0 \text{ if }i\neq j,\text{ and } \iota_i(v_j)=0\text{ for all }j.
\end{equation}
In other words $\iota_i=\frac{\partial}{\partial u_i}$. A computation shows that $\iota_i^2=0$ and $\iota_i\iota_j+\iota_j\iota_i=0$, and it follows that these derivations give the Koszul complex the structure of a dg module over $\Lambda=\Lambda(\iota_1,\ldots,\iota_m)$.

According to \cite[Corollary~4.3.3]{BPbook}, the quotient $S\to k[K]$ provides an algebraic model for the~map
\begin{equation}\label{e_Borel_coh}
    H^*(BT^m;k) \longrightarrow H^*_{T^m}(\mathcal{Z}_K;k)
\end{equation}
induced by the Borel fibration $ET^m \times_{T^m} \mathcal{Z}_K \to BT^m$. Moreover, by {\cite[Theorem~4.5.4]{BPbook}}, the homology of the Koszul complex of $k[K]$ computes the cohomology ring $H^\ast(\mathcal{Z}_K;k)$ of the corresponding moment-angle complex; see \cref{dgiso} below for a more precise statement. In the next section we will use this to describe concretely a (cochain-level) model for the $\Lambda$-module structure on $H^*(\mathcal{Z}_K;k)$, which is in some sense Koszul dual to \eqref{e_Borel_coh}.

Given two subsets $I,J\subseteq [m]$ we frequently use the notation
\[
\varepsilon(I,J) = |\{ (i,j) \in I\times J : i>j \}|,
\]
and when $I=\{i_0,\ldots, i_{n}\}$, we also use the short-hand $\varepsilon(i_0\ldots i_{n},J)=\varepsilon(I,J)$. As we will see below, the signs $(-1)^{\varepsilon(I,J)}$ occur often in the combinatorics of simplicial complexes.

\section{\texorpdfstring{$\Lambda$}{Lambda}-module models for \texorpdfstring{$\mathcal{Z}_K$}{ZK}}
\label{s_exterior_module}

In this section we study exterior algebra module structures on the cochain complex and cohomology of a moment-angle complex $\mathcal{Z}_K$ which are induced by the standard torus action on $\mathcal{Z}_K$. We review two well-known tools for computing the cohomology of a moment-angle complex: the Koszul complex of the Stanley--Reisner ring $k[K]$, and Hochster's formula. Regarding these as algebraic and combinatorial models for $\mathcal{Z}_K$, respectively, we show that both models can be equipped with a compatible structure of a differential graded module over the exterior algebra $\Lambda$. 

Throughout, $k$ continues to denote a field. All cochain and cohomology groups are taken with coefficients in $k$, and all undecorated tensor products are taken over $k$.

\subsection{Cellular decomposition of \texorpdfstring{$\mathcal{Z}_K$}{ZK}}

We begin by recalling a convenient cellular decomposition of $\mathcal{Z}_K$ from \cite[\S 4.4]{BPbook}. Consider the subdivision of the unit disk $D^2=\{z\in\mathbb{C} : |z| \leqslant 1\}$ into three cells where the $0$-skeleton is given by the point $z=1$, the $1$-skeleton is given by the boundary circle $S^1\subset D^2$, and the $2$-skeleton is given by the disk itself. Denoting the $0$-, $1$- and $2$-cell by $1$, $\mathbb{S}$ and $\mathbb{D}$, respectively, and taking products yields a cellular decomposition of the polydisk $(D^2)^m$ with cells parametrised by words $\varkappa \in \{1,\mathbb{S},\mathbb{D}\}^m$ or, equivalently, pairs of subsets $I,J\subseteq [m]$ with $I\cap J=\varnothing$. To each such pair is associated the cell
\begin{align*}
\varkappa(I,J) = \big\{ (z_1,\ldots,z_m)\in (D^2)^m \,:\;\,  &z_i\in \mathbb{D} \text{ for } i\in I, \\
&z_j\in \mathbb{S} \text{ for } j\in J, \\
&z_k=1 \text{ for } k\notin I\cup J \big\}.
\end{align*}
Then for any simplicial complex $K$ on the vertex set $[m]$, $\mathcal{Z}_K$ is the cellular subcomplex of $(D^2)^m$ consisting of those cells $\varkappa(I,J)$ with $I\in K$. 

The cellular cochain complex $\mathcal{C}^*_{\mathrm{cw}}(\mathcal{Z}_K)$ has a basis consisting of cochains $\varkappa(I,J)^\vee$ dual to the cells above with $I\in K$, $J\subseteq [m]$ and $I\cap J=\varnothing$. Although the cellular cochain complex in general does not come with a functorial associative multiplication, Baskakov, Buchstaber and Panov \cite{BBP,BP} have constructed a cellular approximation $\widetilde{\Delta}_K$ to the diagonal map $\Delta_K\colon \mathcal{Z}_K \to \mathcal{Z}_K\times \mathcal{Z}_K$ which is functorial in $K$ and defines a cup product
\[ \cup\colon \mathcal{C}^*_{\mathrm{cw}}(\mathcal{Z}_K)\otimes \mathcal{C}^*_{\mathrm{cw}}(\mathcal{Z}_K) \xrightarrow{\mathmakebox[2em]{\times}} \mathcal{C}^*_{\mathrm{cw}}(\mathcal{Z}_K\times \mathcal{Z}_K) \xrightarrow{\mathmakebox[2em]{\widetilde{\Delta}_K}} \mathcal{C}^*_{\mathrm{cw}}(\mathcal{Z}_K) \]
giving $\mathcal{C}^*_{\mathrm{cw}}(\mathcal{Z}_K)$ the structure of a commutative differential graded algebra.

\subsection{The reduced Koszul complex}\label{s_reduced_Koszul} 

The \emph{reduced Koszul complex} of $k[K]$ is the quotient of the Koszul complex by the acyclic multigraded ideal of elements of non-squarefree multidegree:
\begin{equation}\label{e_red_Koszul}
R(K)= \bigl( \Lambda(u_1,\ldots,u_m) \otimes
 k[K]\bigr) / (v_i^2, v_iu_i),\quad d(u_i)=v_i.
\end{equation}
The differential of the Koszul complex induces a well-defined differential on $R(K)$, and moreover the quotient map
\begin{equation}\label{e_Koszul_qis}
\bigl( \Lambda(u_1,\ldots,u_m) \otimes
 k[K], d\bigr) \xrightarrow{\ \simeq\ } R(K)
\end{equation}
is a quasi-isomorphism of dg algebras, see \cite[Lemma 3.2.6]{BPbook}. In particular,
\begin{equation}\label{e_Koszul_tor}
H_*\bigl( \Lambda(u_1,\ldots,u_m) \otimes
 k[K], d\bigr) \cong H_*(R(K)) \cong \operatorname{Tor}^S_\ast(k[K],k).
\end{equation}
The quotient map \eqref{e_Koszul_qis} is an isomorphism when restricted to any squarefree multidegree. Making this identification, for any $U\subseteq [m]$ we write
\[
R(K)_{i,U} \coloneqq \bigl( k[K] \otimes \Lambda^i(u_1,\ldots,u_m) \bigr)_U.
\]
In particular, we may consider $R(K)$ as a subcomplex of the Koszul complex
\begin{equation}\label{e_Koszul_qis_sub}
 R(K)\xhookrightarrow{\ \simeq\ } \bigl( \Lambda(u_1,\ldots,u_m) \otimes
 k[K], d\bigr),
\end{equation}
by including all elements with squarefree multidegree. Under this inclusion, $R(K)$ is a dg $\Lambda$-submodule of the Koszul complex, using the action $\iota_i=\frac{\partial}{\partial u_i}$ defined in \eqref{e_lambda_action}. We note that this inclusion is not a dg algebra homomorphism, and conversely the quotient \eqref{e_Koszul_qis} is not a dg $\Lambda$-module homomorphism. 

In the next lemma we use the total (cohomological) grading $R^n(K)=\bigoplus_{n=2|U|-i}R(K)_{i,U}$.

\begin{lemma}[{\cite[Lemma~4.5.3]{BPbook}}] \label{dgiso}
There is an isomorphism of differential graded algebras
\begin{align*}
g\colon R(K) &\longrightarrow \mathcal{C}^*_{\mathrm{cw}}(\mathcal{Z}_K) \\
v_Iu_J &\longmapsto \varkappa(I,J)^\vee
\end{align*}
which is functorial in $K$.
\end{lemma}

The action of $T^m$ on $\mathcal{Z}_K$ equips the cohomology ring $H^*(\mathcal{Z}_K)$ with the structure of a graded module over an exterior algebra $\Lambda=\Lambda(\iota_1,\ldots,\iota_m)$ on $m$ generators of degree $-1$. This can be lifted to an action of $\Lambda$ on the cellular cochains of $\mathcal{Z}_K$ as follows.

Regarding $S^1$ as a CW-complex with the same $0$-cell and $1$-cell as $D^2$, we have identifications $\mathcal{C}^*_{\mathrm{cw}}(S^1)=H^*(S^1)=\Lambda(u)$, $|u|=1$. Taking products yields a cellular decomposition of the torus which makes $T^m=(S^1)^m$ a cellular subcomplex of $(D^2)^m$. With respect to these cell structures, it is easy to see that the standard action $\mu\colon T^m\times \mathcal{Z}_K \to \mathcal{Z}_K$ is a cellular map and hence induces a map on cellular cochains.

For each $j\in [m]$, let $\mu_j\colon S^1_j\times \mathcal{Z}_K \to \mathcal{Z}_K$ be the coordinate circle action obtained by restricting the $T^m$-action to the $j^\mathrm{th}$ factor. Identifying $\mathcal{C}^*_{\mathrm{cw}}(S^1_j\times \mathcal{Z}_K)$ with $\Lambda(u_j)\otimes \mathcal{C}^*_{\mathrm{cw}}(\mathcal{Z}_K)$, each coordinate circle action $\mu_j$ induces a map
\begin{equation} \label{action}
\begin{aligned}
\mu_j^*\colon \mathcal{C}^*_{\mathrm{cw}}(\mathcal{Z}_K) &\longrightarrow \Lambda(u_j)\otimes \mathcal{C}^*_{\mathrm{cw}}(\mathcal{Z}_K) \\
\alpha \; &\longmapsto \, 1\otimes\alpha + u_j\otimes \iota_j\alpha,
\end{aligned}
\end{equation}
which defines a graded derivation 
\[ \iota_j\colon \mathcal{C}^*_{\mathrm{cw}}(\mathcal{Z}_K) \longrightarrow \mathcal{C}^{*-1}_{\mathrm{cw}}(\mathcal{Z}_K) \] 
for each $j\in [m]$. Since $[\iota_j,d]=\iota_jd+d\iota_j=0$, these maps induce graded derivations on the cohomology ring $H^*(\mathcal{Z}_K)$ having degree~$-1$, which we call \emph{primary cohomology operations} and also denote by $\iota_1,\ldots,\iota_m$. Moreover, since $[\iota_i,\iota_j]=\iota_i\iota_j+\iota_j\iota_i=0$ for all $i,j\in [m]$, the actions of these derivations extend to graded $\Lambda$-module structures on both $\mathcal{C}^*_{\mathrm{cw}}(\mathcal{Z}_K)$ and $H^*(\mathcal{Z}_K)$.  

\begin{remark}
Since $\Lambda$ acts on the cohomology of $\mathcal{Z}_K$ for every simplicial complex $K$ on the vertex set $[m]$, and this action is functorial in $K$, we think of $\Lambda$ as an algebra of cohomology operations for moment-angle complexes. In \cref{s_HB_pert} we will embed $\Lambda$ into a larger algebra of \emph{higher cohomology operations}.
\end{remark}

\begin{lemma} \label{cochain_model}
The map $g\colon R(K) \to \mathcal{C}^*_{\mathrm{cw}}(\mathcal{Z}_K)$ of Lemma~\ref{dgiso} is an isomorphism of differential graded $\Lambda(\iota_1,\ldots,\iota_m)$-modules.
\end{lemma}

\begin{proof}
Since the generators of $\Lambda$ act by derivations, it suffices by \cref{dgiso} to show that $g\circ\iota_i$ and $\iota_i\circ g$ agree on the generators of $R(K)$ for each $i=1,\ldots,m$. By definition of the $\Lambda$-module structure on $R(K)$, $g\circ\iota_i$ vanishes on all generators except $u_i$, in which case $(g\circ\iota_i)(u_i)=g(1)=\varkappa(\varnothing,\varnothing)^\vee$, the dual of the $0$-cell $(1,\ldots,1) \in \mathcal{C}_0^{\mathrm{cw}}(\mathcal{Z}_K;k)$. On the other hand, we have
\[
(\iota_i\circ g)(v_j)=\iota_i \varkappa(\{j\},\varnothing)^\vee \quad\text{and}\quad (\iota_i\circ g)(u_j)=\iota_i \varkappa(\varnothing,\{j\})^\vee.
\]
Since $v_j\in R^2(K)$ is closed for all $j$, so is $\iota_i \varkappa(\{j\},\varnothing)^\vee$. It follows that $\iota_i \varkappa(\{j\},\varnothing)^\vee=0$ since the only cocycle in $\mathcal{C}^1_{\mathrm{cw}}(\mathcal{Z}_K)$ is trivial by \cref{dgiso}. To see that $\iota_i \varkappa(\varnothing,\{j\})^\vee$ equals $\varkappa(\varnothing,\varnothing)^\vee$ for $i=j$ and is trivial for $i\neq j$, observe that the orbit of the $0$-cell $(1,\ldots,1)$ under the action $\mu_i\colon S^1_i\times \mathcal{Z}_K \to \mathcal{Z}_K$ is exactly the $1$-cell $\varkappa(\varnothing,\{i\})$. Dualizing, it follows that
\[ \mu_i^*\bigl(\varkappa(\varnothing,\{j\})^\vee\bigr) =
\begin{cases}
1\otimes \varkappa(\varnothing,\{i\})^\vee + u_i\otimes \varkappa(\varnothing,\varnothing)^\vee & \text{if } i=j \\
1\otimes \varkappa(\varnothing,\{j\})^\vee & \text{if } i\neq j.
\end{cases} \]
Therefore by \eqref{action} we have 
\[ \iota_i \varkappa(\varnothing,\{j\})^\vee =
\begin{cases}
\varkappa(\varnothing,\varnothing)^\vee & \text{if } i=j \\
0 & \text{if } i\neq j,
\end{cases} \]
which completes the proof.
\end{proof}

\subsection{Combinatorial model for \texorpdfstring{$\mathcal{C}^*_{\mathrm{cw}}(\mathcal{Z}_K)$}{C(ZK)}}

The following celebrated result of Hochster interprets the Koszul homology of a Stanley--Reisner ring $k[K]$ in terms of the simplicial cohomology groups of full subcomplexes of $K$.

\begin{theorem}[Hochster's formula \cite{H}]
For each squarefree multidegree $U\subseteq [m]$, there is an isomorphism of cochain complexes
\begin{equation} \label{Hoc_map}
\begin{aligned}
h\colon \bigl( k[K] \otimes \Lambda^*(u_1,\ldots,u_m) \bigr)_U &\longrightarrow \widetilde{C}^{|U|-*-1}(K_U) \\
v_Iu_J \;(\text{with } I\cup J=U) &\longmapsto (-1)^{\varepsilon(I,U)}I^\vee. 
\end{aligned}
\end{equation}
\end{theorem}

Note that, summing over all $U\subseteq [m]$, the theorem above gives an isomorphism of complexes $R(K) \cong \bigoplus_{U\subseteq [m]} \widetilde{C}^*(K_U)$. Taking cohomology then yields the usual Hochster decomposition
\[
\operatorname{Tor}^S_*(k[K],k) \cong \bigoplus_{U\subseteq [m]} \widetilde{H}^{|U|-*-1}(K_U).
\]
The next result describes the $\Lambda$-module structures on $R(K)\cong \mathcal{C}^*_{\mathrm{cw}}(\mathcal{Z}_K)$ and $H^*(\mathcal{Z}_K)$ in terms of the decompositions above.

\begin{lemma} \label{comm_diagrams}
For each $j\in [m]$ and each squarefree multidegree $U\subseteq [m]$ containing~$j$, there are commutative diagrams
\[
\xymatrix{
R(K)_{i,U} \ar[r]^-h \ar[d]^{\iota_j} & \widetilde{C}^{|U|-i-1}(K_U) \ar[d] & \omega \ar@{|->}[d] \\
R(K)_{i-1,U\smallsetminus j} \ar[r]^-h & \widetilde{C}^{|U|-i-1}(K_{U\smallsetminus j}) & (-1)^{\varepsilon(j,U)+|U|-i}\omega|_{U\smallsetminus j}
}
\]
where the horizontal maps are the isomorphisms of \eqref{Hoc_map}, and the right vertical map is defined by the inclusion $K_{U\smallsetminus j} \hookrightarrow K_U$ (up to the indicated sign). In particular, there are commutative diagrams
\[
\xymatrix{
H^n(\mathcal{Z}_K) \ar[r]^-\cong \ar[d]^{\iota_j} & \displaystyle\bigoplus_{U\subseteq [m]} \widetilde{H}^{n-|U|-1}(K_U) \ar[d]  \\
H^{n-1}(\mathcal{Z}_K) \ar[r]^-\cong & \displaystyle\bigoplus_{U\subseteq [m]} \widetilde{H}^{n-|U|-2}(K_U)
}
\]
where the right vertical map is a sum of maps induced by inclusions $K_{U\smallsetminus j} \hookrightarrow K_U$ for $j\in U$ (and is trivial on summands indexed by $U\subseteq [m]$ with $j\notin U$).
\end{lemma}

\begin{proof}
Let $v_Iu_J\in R(K)_{i,U}$ (so $I\in K$, $J\in [m]$, $I\cap J=\varnothing$ and $I\cup J=U$) and assume $j\in U$. If $j\in I$, then $\iota_j(v_Iu_J)=0$, and $h(v_Iu_J)=(-1)^{\varepsilon(I,U)}I^\vee$ is in the kernel of the restriction map $\widetilde{C}^*(K_U)\to \widetilde{C}^*(K_{U\smallsetminus j})$ since, dually, $I\in \widetilde{C}_*(K_U)$ is clearly not in the image of $\widetilde{C}_*(K_{U\smallsetminus j})\to \widetilde{C}_*(K_U)$ when $j\in I$. On the other hand, if $j\in J$, then
following anticlockwise around the first diagram, we have
\begin{align}
h\iota_j(v_Iu_J) &= h\bigl( (-1)^{\varepsilon(j,J)} v_Iu_{J\smallsetminus j} \bigr) \nonumber \\
&= (-1)^{\varepsilon(j,J)}(-1)^{\varepsilon(I,U\smallsetminus j)} I^\vee. \label{sign1}
\end{align}
Following clockwise around the diagram, the right vertical map sends $h(v_Iu_J)=(-1)^{\varepsilon(I,U)}I^\vee$ to
\begin{equation} \label{sign2}
(-1)^{\varepsilon(I,U)}(-1)^{\varepsilon(j,U)+|U|-|J|}I^\vee=(-1)^{\varepsilon(I,U)}(-1)^{\varepsilon(j,U)+|I|}I^\vee.
\end{equation}
To see that \eqref{sign1} and \eqref{sign2} are equal, observe that
\[
\underbrace{\varepsilon(j,U)-\varepsilon(j,J)}_{\varepsilon(j,I)} + \underbrace{\varepsilon(I,U)-\varepsilon(I,U\smallsetminus j)}_{\varepsilon(I,j)} + |I| = |I|+|I| \equiv 0\text{ mod } 2,
\]
so the signs agree.

The second commutative diagram in the statement of the lemma is obtained from the first by passing to cohomology, summing over all $i\in\mathbb{Z}$ and $U\subseteq [m]$ with $2|U|-i=n$, and absorbing signs into the horizontal isomorphisms.
\end{proof}

\begin{remark} \label{mod_structures}
By Lemma~\ref{comm_diagrams}, $\bigoplus_{U\subseteq [m]} \widetilde{C}^*(K_U)$ and hence $\bigoplus_{U\subseteq [m]} \widetilde{H}^*(K_U)$ can clearly be given $\Lambda(\iota_1,\ldots,\iota_m)$-module structures by letting $\iota_j$ act on simplicial cochains by the right vertical map in the diagram. With respect to these module structures, the Hochster decompositions 
\[ R(K)\cong \bigoplus_{U\subseteq [m]} \widetilde{C}^*(K_U) \quad\text{ and }\quad \operatorname{Tor}^S(k[K],k) \cong \bigoplus_{U\subseteq [m]} \widetilde{H}^*(K_U) \]
become isomorphisms of differential graded $\Lambda(\iota_1,\ldots,\iota_m)$-modules.
\end{remark}

Combining Lemma~\ref{cochain_model}, Lemma~\ref{comm_diagrams} and Remark~\ref{mod_structures}, we obtain the following.

\begin{corollary}\label{c_isom_Tor_Zk}
There is a zig-zag of isomorphisms of differential graded $\Lambda(\iota_1,\ldots,\iota_m)$-modules
\[ \mathcal{C}^*_{\mathrm{cw}}(\mathcal{Z}_K) \stackrel{\;g}{\longleftarrow} R(K) \stackrel{h\;}{\longrightarrow} \bigoplus_{U\subseteq [m]} \widetilde{C}^*(K_U), \]
inducing graded $\Lambda(\iota_1,\ldots,\iota_m)$-module isomorphisms
\[ H^*(\mathcal{Z}_K) \cong \operatorname{Tor}^S(k[K],k) \cong \bigoplus_{U\subseteq [m]} \widetilde{H}^*(K_U). \]
\end{corollary}

\subsection{Higher cohomology operations}\label{s_higher_ops}

To any differential graded module $N$ over $\Lambda=\Lambda(\iota_1,\ldots,\iota_m)$, one can associate a family of \emph{higher cohomology operations} acting on $H^*(N)$,  indexed by monomials:
\[
\delta_s, \quad s=v_1^{n_1}\cdots v_m^{n_m}\in S=k[v_1,\ldots,v_m].
\] 
 In the context of Lie group actions and equivariant cohomology, these operations were introduced by Goresky, Kottwitz and MacPherson as obstructions to equivariant formality in~\cite{GKM}; see also  \cref{r_infinity_nonsense,r_infinity_nonsense_2} for other contexts in which analogous operations arise. For any monomial $s=v_1^{n_1}\cdots v_m^{n_m}$, the operation $\delta_s$ is of degree~$1-2\sum_{j=1}^m n_j$ and is well-defined as a map of the form 
\begin{equation}\label{e_indeterminacy}
\bigcap_{t|s, t\neq s} \operatorname{ker}\delta_t \longrightarrow \frac{H^*(N)}{\sum_{t|s, t\neq s}\operatorname{im}\delta_t},
\end{equation}
see \cite[Proposition~13.8]{GKM}. In particular, the \emph{primary operations} $\delta_{v_j}$ are the degree~$-1$ maps
\begin{align*} 
\delta_{v_j}\colon H^*(N) &\longrightarrow H^{*-1}(N) \\
[\alpha] &\longmapsto [\iota_j\alpha]
\end{align*}
induced by the $\Lambda$-module structure on $N$. 

If $[\alpha]\in H^n(N)$ and $\delta_{v_i}[\alpha]=\delta_{v_j}[\alpha]=0$, then $\iota_i\alpha=d\beta_i$ and $\iota_j\alpha=d\beta_j$ for some $\beta_i,\beta_j\in N^{n-2}$. Note that
\begin{align*}
d(\iota_i\beta_j+\iota_j\beta_i) &= -\iota_id\beta_j-\iota_jd\beta_i \\
&= -\iota_i\iota_j\alpha-\iota_j\iota_i\alpha \\
&= 0
\end{align*}
since $\iota_i\iota_j+\iota_j\iota_i=0$. In this case, the \emph{secondary operation} $\delta_{v_iv_j}$, of degree~$-3$, is defined by
\begin{equation} \label{sec_op}
\delta_{v_iv_j}[\alpha] = [\iota_i\beta_j+\iota_j\beta_i]. 
\end{equation}

Assuming the primary operations $\delta_{v_i}$, $\delta_{v_j}$, $\delta_{v_k}$ and secondary operations $\delta_{v_iv_j}$, $\delta_{v_iv_k}$, $\delta_{v_jv_k}$ act trivially on $H^*(N)$, the \emph{tertiary operation} $\delta_{v_iv_jv_k}\colon H^*(N) \to H^{*-5}(N)$ is defined by
\[ \delta_{v_iv_jv_k}[\alpha] = [\iota_i\beta_{jk}+\iota_j\beta_{ik}+\iota_k\beta_{ij}], \]
where
\[
\begin{aligned}
d\beta_{ij}&=\iota_i\beta_j+\iota_j\beta_i \\
d\beta_{ik}&=\iota_i\beta_k+\iota_k\beta_i \\
d\beta_{jk}&=\iota_j\beta_k+\iota_k\beta_j
\end{aligned}
\qquad\text{ and }\qquad
\begin{aligned}
d\beta_i&=\iota_i\alpha \\
d\beta_j&=\iota_j\alpha \\
d\beta_k&=\iota_k\alpha
\end{aligned}
\]
for some $\beta_i,\beta_j,\beta_k\in N^{n-2}$ and $\beta_{ij},\beta_{ik},\beta_{jk}\in N^{n-4}$.

Similar cochain-level formulas can be given for all higher operations: assuming representatives $\delta_t(\alpha)\in N$ of $\delta_t[\alpha]$ have been defined for all monomials $t$ dividing $s=v_{i_1}\cdots v_{i_q}$ ($i_1,\ldots,i_q$ not necessarily distinct) with $|t|<|s|$, and assuming $\delta_t$ is trivial on $H^*(N)$ for all such $t$, then $\delta_s[\alpha]$ is represented by
\begin{equation} \label{rep}
\sum_{\ell=1}^q\iota_{i_\ell}d^{-1}\delta_{s/v_{i_\ell}}(\alpha) \in N^{n-2q+1},
\end{equation}
where $d^{-1}\delta_{s/v_{i_\ell}}(\alpha)$ denotes a choice of preimage of $\delta_{s/v_{i_\ell}}(\alpha)$ satisfying a certain system of equations analogous to those above.

We will frequently make use of the fact that $\delta_s$ is a well-defined map $H^*(N)\to H^{*-|s|+1}(N)$ with no indeterminacy when the previous operations $\delta_t$ with~$t|s$ all vanish. This follows from~\cite[Proposition~13.8]{GKM}, cited above, which in turn follows from an identification of the higher operations with differentials in a spectral sequence associated to the Koszul dual dg $S$-module $S\otimes N$ with differential $1_S\otimes d+\sum_{j=1}^m v_j\otimes \iota_j$. (For cohomology operations on $H^*(\mathcal{Z}_K)$, this fact will also follow from the proof of Lemma~\ref{tech_lem}, which identifies the higher operations with differentials in a different spectral sequence.) We include here a more direct proof for the special case of a secondary operation, which we will need in Section~\ref{comb_models_sec}.

\begin{lemma}
Let $(N,d)$ be a differential graded $\Lambda(\iota_1,\ldots,\iota_m)$-module. Suppose the primary operations $\delta_{v_i}$ and $\delta_{v_j}$ vanish on $H^*(N)$ for some $i,j\in [m]$. Then the secondary operation \eqref{sec_op} is a well-defined map
\[ \delta_{v_iv_j}\colon H^*(N) \longrightarrow H^{*-3}(N). \]
\end{lemma}

\begin{proof}
Let $[\alpha]\in H^n(N)$. Since $\delta_{v_i}=\delta_{v_j}=0$, the secondary operation $\delta_{v_iv_j}[\alpha]$ is represented by $\iota_i\beta_j+\iota_j\beta_i$ where $\iota_i\alpha=d\beta_i$ and $\iota_j\alpha=d\beta_j$. Suppose $\beta_i',\beta_j'\in N^{n-2}$ also satisfy the equations $\iota_i\alpha=d\beta_i'$ and $\iota_j\alpha=d\beta_j'$. Then because $\beta_i-\beta_i'$ and $\beta_j-\beta_j'$ are closed, there must exist $\omega,\omega'\in N^{n-4}$ with $\iota_i(\beta_j-\beta_j')=d\omega$ and $\iota_j(\beta_i-\beta_i')=d\omega'$ since $\delta_{v_i}[\beta_j-\beta_j']=\delta_{v_j}[\beta_i-\beta_i']=0$. Therefore
\[ (\iota_i\beta_j+\iota_j\beta_i)-(\iota_i\beta_j'+\iota_j\beta_i')=\iota_i(\beta_j-\beta_j')+\iota_j(\beta_i-\beta_i')=d(\omega+\omega'), \]
from which it follows that $[\iota_i\beta_j+\iota_j\beta_i]=[\iota_i\beta_j'+\iota_j\beta_i']\in H^{n-3}(N)$ and $\delta_{v_iv_j}$ is well-defined.
\end{proof}

\begin{remark}
Since the generators of $\Lambda$ act by derivations on the cochains of a moment-angle complex $\mathcal{Z}_K$, it can be shown that $\delta_{v_iv_j}$ is in fact a well-defined derivation of the cohomology ring $H^*(\mathcal{Z}_K)$ when the primary operations $\delta_{v_i}$ and $\delta_{v_j}$ are trivial (although we will not need this). Analogous computations establishing the well-definedness of tertiary and higher operations when lower degree operations vanish are similarly straightforward but tedious.
\end{remark}

Taking $N$ to be $\mathcal{C}^*_{\mathrm{cw}}(\mathcal{Z}_K)$, we next describe how the higher operations interact with the multigrading on $H^*(\mathcal{Z}_K)$. Below we identify $H^*(\mathcal{Z}_K)$ with $\bigoplus_{J\subseteq [m]} \widetilde{H}^*(K_J)$ via the equivalences of \cref{c_isom_Tor_Zk} and denote the multidegree $J$ part of $H^*(\mathcal{Z}_K)$ by $\widetilde{H}^*(K_J)$.

\begin{lemma} \label{l_deg_of_ops}
Let $s=v_{i_1}\cdots v_{i_q}\in S$ be a monomial. Suppose the operation $\delta_t$ acts trivially on $H^*(\mathcal{Z}_K)\cong \bigoplus_{J\subseteq [m]} \widetilde{H}^*(K_J)$ for all $t|s$ with $|t|<|s|$, and consider the operation
\[ \delta_s\colon H^*(\mathcal{Z}_K) \longrightarrow H^{*-2q+1}(\mathcal{Z}_K). \] 
If $s$ is squarefree, then
\[ \delta_s\bigl(\widetilde{H}^p(K_J)\bigr) \subseteq \widetilde{H}^{p-q+1}(K_{J\smallsetminus i_1\cdots i_q}) \]
for all $J\subseteq [m]$ containing $i_1,\ldots,i_q$. Otherwise, $\delta_s=0$.
\end{lemma}

\begin{proof}
The claim holds for primary operations $\delta_{v_j}=\iota_j$ by Lemma~\ref{comm_diagrams}. Let $s=v_{i_1}\cdots v_{i_q}$ be a squarefree monomial and assume inductively that for any $[\alpha]\in \widetilde{H}^p(K_J)$ and $1\leqslant\ell\leqslant q$, a representative $\delta_{s/v_{i_\ell}}(\alpha)$ of the cohomology class $\delta_{s/v_{i_\ell}}[\alpha]$ (as in \eqref{rep}) can be chosen to lie in $\widetilde{C}^{p-q+2}(K_{J\smallsetminus i_1\cdots\widehat{i_\ell}\cdots i_q})$. Then for any $\alpha\in \widetilde{H}^p(K_J)$, $\delta_s[\alpha]$ is represented by a cochain of the form
\[ \sum_{\ell=1}^q\iota_{i_\ell}d^{-1}\delta_{s/v_{i_\ell}}(\alpha)\in \bigoplus_{J\subseteq [m]} \widetilde{C}^*(K_J), \]
where each preimage $d^{-1}\delta_{s/v_{i_\ell}}(\alpha)$ can be chosen to lie in $\widetilde{C}^{p-q+1}(K_{J\smallsetminus i_1\cdots\widehat{i_\ell}\cdots i_q})$. Therefore each term $\iota_{i_\ell}d^{-1}\delta_{s/v_{i_\ell}}(\alpha)$ is of multidegree $J\smallsetminus i_1\cdots i_q$, and $\delta_s[\alpha]\in \widetilde{H}^{p-q+1}(K_{J\smallsetminus i_1\cdots i_q})$, as desired. Finally, if $s$ is not squarefree, the triviality of $\delta_s$ follows from the fact that $\iota_{i_\ell}$ acts trivially on $\widetilde{C}^*(K_U)$ when $j\notin U$. 
\end{proof}

By \cref{l_deg_of_ops}, among the higher cohomology operations induced by the standard $T^m$-action on $\mathcal{Z}_K$, we may restrict attention to those indexed by squarefree monomials in $S=k[v_1,\ldots,v_m]$. (This is in contrast to the case of a general $T^m$-space; see \cref{ex_not_sq_free_1,ex_not_sq_free_2}). Identifying squarefree monomials $s=v_I=v_{i_1}\!\cdots v_{i_q}$ with subsets $I=\{i_1,\ldots,i_q\} \subseteq [m]$, we use the notation $\delta_I$ or $\delta_{i_1\!\cdots i_q}$ to denote $\delta_s$.

In the next section we take a different approach to the higher cohomology operations for moment-angle complexes by rigidifying them into well-defined endomorphisms of $H^*(\mathcal{Z}_K)$. These operations can then be used to describe explicit Hirsch--Brown models for these toric spaces, as has been done in a different context in~\cite{CJ}.

\section{Hirsch--Brown models from perturbation theory} \label{s_HB_pert}

In this section we explain how the minimal free resolution of the Stanley--Reisner ring can be constructed from the reduced Koszul complex using the homological perturbation lemma. As a consequence we obtain direct formulae for the cohomology operations $\delta_I$ from \cref{s_higher_ops}, and we obtain an explicit Hirsch--Brown model for the action of any subtorus $H\subseteq T^m$ on $\mathcal{Z}_K$.

We fix, as usual, a simplicial complex $K$ on a vertex set $[m]$, and we consider the reduced Koszul complex $R(K)$ with its dg $\Lambda$-module structure as in \cref{s_reduced_Koszul}. We need to introduce some notation to be used in the next lemma. The element
\[
\rho={\sum\limits_{i\in [m]}} v_i\otimes \iota_i \in S \otimes \Lambda
\]
defines by left multiplication a homological degree $-1$, multigraded endomorphism
\begin{equation} \label{e_pert}
\rho\colon S\otimes R(K)\longrightarrow S\otimes R(K).
\end{equation}
A direct computation shows that $(S\otimes d +\rho)^2=0$, where $d$ is the differential on $R(K)$. In other words, $\rho$ defines a perturbation of the complex $S\otimes R(K)$.

We recall that $R(K)$ admits a basis of monomials $v_Vu_U$ indexed by disjoint subsets $U,V\subseteq [m]$ with $V\in K$, and we use this to define a linear map
\begin{equation}\label{e_theta}
\theta\colon R(K) \to k[K],\quad v_V\mapsto v_V\ \text{ and }\ v_Vu_U \mapsto 0\, \text{ if }\, U\neq \varnothing.
\end{equation}
In the next lemma $\theta_S\colon S\otimes R(K) \to k[K]$ is the unique $S$-linear extension of $\theta$.

\begin{lemma}\label{l_theta_pert}
If $S\otimes R(K)$ is equipped with the perturbed differential $S\otimes d +\rho$, then the map $\theta_S\colon S\otimes R(K) \to k[K]$ is a quasi-isomorphism of dg $S$-modules.
\end{lemma}

\begin{proof}
We use the quasi-isomorphism
\[
 \theta'\colon R(K)\xhookrightarrow{\ \simeq\ } \bigl( \Lambda(u_1,\ldots,u_m) \otimes
 k[K], d\bigr)
\]
of dg $\Lambda$-modules from \eqref{e_Koszul_qis_sub} above. From this we obtain a chain map
\[
\theta'_S\colon \big(S\otimes R(K), S\otimes d +\rho\big) \to \big(S\otimes\Lambda(u_1,\ldots,u_m)\otimes k[K], S\otimes d +\rho \big),
\]
giving both sides the differential perturbed by $\rho$. Since $\theta'$ is a quasi-isomorphism and $\theta'= k\otimes_S\theta'_S$, it follows that $\theta'_S$ is a quasi-isomorphism by the (derived) Nakayama Lemma. 

We also make use of the surjective chain map
\[
\theta''_S\colon\big(S\otimes\Lambda(u_1,\ldots,u_m)\otimes k[K], S\otimes d +\rho \big) \to k[K], \quad f\otimes 1 \otimes g \mapsto fg,\, f\otimes u_U \otimes g \mapsto 0  \text{ if } U\neq \varnothing.
\]
It is well known that
\[
\theta''_S\otimes_{k[K]} k \colon\big(S\otimes\Lambda(u_1,\ldots,u_m),\rho \big) \to k
\]
is a quasi-isomorphism, as the left-hand side is the standard Koszul complex on the maximal homogeneous ideal of $S$. It follows that $\theta''_S$ is a quasi-isomorphism, again by the (derived) Nakayama Lemma.  Finally, $\theta_S$ factors as
\[
S\otimes R(K) \xrightarrow{\theta'_S} S\otimes \Lambda(u_1,\ldots,u_m)\otimes k[K] \xrightarrow{\theta''_S}  k[K].
\]
From this we see that $\theta_S$ is a chain map and that it is a quasi-isomorphism.
\end{proof}

\begin{remark}
The perturbed differentials appearing above can be interpreted within the framework of \emph{twisted tensor products}, using the \emph{twisting cochain} $\tau\colon \Lambda(u_1,\ldots,u_m) \to S$, with $u_i\mapsto v_i$ and $u_U\mapsto 0$ if $|U|\neq 1$; consult \cite{LVbook}
for more on this theory. Even more classically, the perturbed (or twisted) tensor products $S\otimes - $ and $\Lambda(u_1,\ldots,u_m)\otimes -$ that we have used are the well-known BGG functors between dg $\Lambda$-modules and dg $S$-modules, realising the famous Bernstein--Gel'fand--Gel'fand  correspondence \cite{BGG}.
\end{remark}

Recall from \eqref{e_Koszul_tor} that the reduced Koszul complex satisfies $H(R(K))\cong  \operatorname{Tor}^S(k[K],k)$. In order to construct cohomology operations explicitly as well-defined endomorphisms of $\operatorname{Tor}^S(k[K],k)$, we fix data that realises this isomorphism at the chain level.

\begin{definition}\label{def_DRdata}
A \emph{multigraded deformation retraction} for $R(K)$ is a diagram of maps, homogeneous with respect to multidegree,
\begin{equation}\label{e_h_retract_data}
\begin{tikzcd}
 \operatorname{Tor}^S(k[K],k) \ar[shift right= 1mm,r,"\sigma"'] & R(K) \ar[shift right= 1mm,l,"\pi"'] \ar[loop right,"h"]
\end{tikzcd}
\end{equation}
such that $\pi$ and $\sigma$ are degree zero chain maps satisfying  $\pi\sigma=1$, and $h$ is a (homological) degree~$1$ homotopy with $dh+hd = 1-\sigma\pi$. 
\end{definition}

One can readily construct multigraded deformation retraction data: Within $R(K)$ denote $Z=\ker(d)$ and $B={\rm im}(d)$, and make the identification  $\operatorname{Tor}^S(k[K],k)=  Z/B$. Next choose a splitting $s_1\colon R(K)/Z \to R(K)$ of the surjection $p_1\colon R(K) \to R(K)/Z$, and a splitting $s_2\colon Z/B \to Z$ of the surjection $p_2\colon Z\to Z/B$, both as $ \mathbb{Z}\times \mathbb{Z}^m$ graded vector spaces. Then we may take $\sigma= s_2$ and $\pi = p_2(1-s_1p_1)$, and since $d s_1 \colon R(K)/Z\to B$ is an isomorphism we may take $h=s_1(d s_1)^{-1}$. These maps all together satisfy the conditions of \cref{def_DRdata}.

\begin{theorem}\label{t_coh_ops_perturb}
Let $K$ be a simplicial complex on the vertex set $[m]$ and choose a multigraded deformation retraction for the reduced Koszul complex $R(K)$, as in (\ref{e_h_retract_data}). For each $I\subseteq [m]$ we define an operation on $\operatorname{Tor}^S(k[K],k)$ having homological degree $-1$ and multidegree $-I$ by the formula
\[
\partial_{I} = \sum_{I=\{i_1,...,i_n\}} \pi \iota_{i_1}h \iota_{i_2} h \cdots h \iota_{i_n} \sigma,
\]
where the sum is taken over all possible orderings of $I$ (therefore having $|I|!$ summands). These operations are the coefficients of the differential in the minimal (multigraded) free resolution of the Stanley--Reisner ring. In other words, there is a quasi-isomorphism of dg $S$-modules
\[
\big(S\otimes \operatorname{Tor}^S(k[K],k) ,d \big)\xrightarrow{\ \simeq\ } k[K] \quad \text{ with } d = \sum_{I\subseteq [m]} v_{I}\otimes \partial_{I}.
\]
\end{theorem}

We have used the notation $\partial_I$ to distinguish these operations from the operations $\delta_I$ defined in \cref{s_higher_ops}; however, we will show in \cref{p_operations_from_pert} that they are essentially the same.

\begin{proof}
To begin with, we extend the chosen deformation retraction $S$-linearly:
\[
\begin{tikzcd}
 S\otimes \operatorname{Tor}^S(k[K],k) \ar[shift right= 1mm,r,"\sigma_S"'] & S\otimes R(K) \ar[shift right= 1mm,l,"\pi_S"'] \ar[loop right," h_S"]
\end{tikzcd}
\]
where $\sigma_S=S\otimes\iota$, $\pi_S=S\otimes\pi$, $h_S=S\otimes h$, and the complex $ S\otimes R(K)$ has the differential $S\otimes d$. This is again a deformation retraction.

We make use of the homological perturbation lemma, whose history goes back at least to~\cite{GUG}; a convenient reference is \cite{Crainic}. 
Using the perturbation (\ref{e_pert}) and \cite[2.4]{Crainic} we obtain the perturbed deformation retraction
\[
\begin{tikzcd}
 \big(S\otimes \operatorname{Tor}^S(k[K],k) ,d'\big)\ar[shift right= 1mm,r,"\sigma_S'"'] & \big(S\otimes R(K), S\otimes d +\rho\big) \ar[shift right= 1mm,l,"\pi_S'"'] \ar[loop right," h_S'"]
\end{tikzcd}
\]
where 
\[
\begin{tikzcd}[row sep=0mm,column sep=5mm]
 d'=  \sum_{n\geqslant 0}\pi_S\rho (h_S\rho)^n\sigma_S, & \pi_S'=  \pi_S+\sum_{n\geqslant 0}\pi_S\rho (h_S\rho)^nh_S,\\
  \sigma_S'=  \sigma_S+\sum_{i\geqslant 0}h_S\rho (h_S\rho)^n\sigma_S, & h_S'=  h_S+\sum_{n\geqslant 0}h_S\rho (h_S\rho)^nh_S.
\end{tikzcd}
\]
In the given formula for $d'$, each of $\pi_S,\sigma_S$ and $h_S$ are $S$-linear extensions of maps of vector spaces, and therefore all non-constant $S$-coefficients in $d'$ arise from $\rho=\sum_{i\in[m]}v_i\otimes \iota_i$:
\[
d'=  \sum_{n\geqslant 0}\sum_{i_1,\ldots,i_n\in[m]}(v_{i_1}\cdots v_{i_n})\pi_S\iota_{i_1} h_S\iota_{i_2}h_S\cdots h_S\iota_{i_n}\sigma_S.
\]
Since $R(K)$ is nonzero only in squarefree multidegree, and since $h_S$ has multidegree zero while each $\iota_i$ has multidegree $(0,\ldots,-1,\ldots,0)$, the summands in the above series can only be nonzero when  $i_1,\ldots,i_n$ are distinct. In other words, the summands are indexed by subsets $I=\{i_1,\ldots,i_n\}$, with each ordering contributing a different summand. This yields the expression $d' = \sum_{I\subseteq [m]} v_{I}\otimes \partial_{I}$.

Finally, using $\theta_S$ from (\ref{e_theta}), the natural projection
\[
\theta_S\sigma_S'=\theta_S\sigma_S\colon  \big(S\otimes \operatorname{Tor}^S(k[K],k) ,d'\big) \longrightarrow k[K]
\]
is a quasi-isomorphism by \cref{l_theta_pert}.
\end{proof}

\begin{remark}\label{r_infinity_nonsense}
The system of cohomology operations on $\operatorname{Tor}^S(k[K],k)$ constructed in \cref{t_coh_ops_perturb} contains exactly the data needed to recover the derived tensor product $k[K]\otimes^{\rm L}_Sk$ up to a quasi-isomorphism of dg $\Lambda$-modules. In this sense, the cohomology operations play a similar role to A$_\infty$-module structures. In fact, the operations $\{\partial_I\}$ can be obtained by symmetrising the higher multiplications $\{m_n\}$ in a choice of A$_\infty$-$\Lambda$-module structure on $\operatorname{Tor}^S(k[K],k)$. Continuing along these lines, it is possible to consider the data $\{\partial_I\}$ as a kind of $\infty$-$\Lambda$-module structure with respect to the twisting cochain $\tau\colon S^\vee\to \Lambda$; cf.\ \cite{LVbook}.
\end{remark}

\begin{remark}\label{r_infinity_nonsense_2}
The same operations are also closely related to the systems of higher homotopies defined by Eisenbud, on resolutions of modules over local complete intersection rings \cite{EISENBUD}. More precisely, if $A$ is a local ring with residue field $k$, and $B=A/(f_1,\ldots,f_m)$ is the quotient by a regular sequence, then for any $B$-module $M$, there is a natural $\Lambda(x_1,\ldots,x_m)$-module structure on $\operatorname{Tor}^A(M,k)$, with $x_i$ having homological degree $1$. A system of higher homotopies on the minimal $A$-free resolution of $M$ (as in \cite{EISENBUD}) induces a system of cohomology operators on $\operatorname{Tor}^A(M,k)$ extending its $\Lambda(x_1,\ldots,x_m)$-module structure, analogous to those in \cref{t_coh_ops_perturb} but having different degrees.
\end{remark}

These cohomology operations can be interpreted in terms of the equivariant topology of the moment-angle complex $\mathcal{Z}_K$. This time we choose a multigraded deformation retraction for the cellular cochain algebra
\begin{equation}\label{e_h_retract_data_Zk}
\begin{tikzcd}
 H^*(\mathcal{Z}_K;k) \ar[shift right= 1mm,r,"\sigma"'] & \mathcal{C}^*_{\mathrm{cw}}(\mathcal{Z}_K;k) \ar[shift right= 1mm,l,"\pi"'] \ar[loop right,"h"]
\end{tikzcd}
\end{equation}
satisfying the same conditions as in \eqref{e_h_retract_data}. Using the isomorphism $R(K)\cong \mathcal{C}^*_{\mathrm{cw}}(\mathcal{Z}_K;k)$ of dg $\Lambda$-modules from \cref{cochain_model}, the statement of \cref{t_coh_ops_perturb} translates to the corollary below. We first recall a definition from rational homotopy theory and the theory of transformation groups (generalised slightly to allow for field coefficients other than $k=\mathbb{Q}$). 

For a $T$-space $X$, although the $S=H^*(BT;k)$-module structure on $H^*_T(X;k)=H^*(ET\times_T X;k)$ cannot in general be lifted to an action of $S$ on $C^*(ET\times_T X;k)$, the existence of a dg algebra quasi-isomorphism $f\colon C^*(BT;k)\to H^*(BT;k)$ identifies the derived categories of dg $C^*(BT;k)$-modules and dg $H^*(BT;k)$-modules. Explicitly, the functor
\begin{equation} \label{derived_equiv}
- \otimes_{C^*(BT^m)}^{\rm L} H^*(BT)\colon D(C^*(BT)) \longrightarrow D(H^*(BT))
\end{equation}
is an equivalence of categories with quasi-inverse given by the restriction of scalars along $f$. Below, we identify $C^*(ET\times_T X;k)$ with its image under the equivalence above.

\begin{definition}
The \emph{minimal Hirsch--Brown model} of the action of a torus $T$ on a space $X$ is the dg $S$-module minimal model of $C^*(ET\times_T X;k)$, where $S=H^*(BT;k)$.
\end{definition}

The minimal Hirsch--Brown model is therefore a semi-free dg $S$-module $(S\otimes V,d)$ that satisfies the minimality condition $\operatorname{im}(d)\subseteq S^{>0}\otimes V$ and is quasi-isomorphic to (the image under~\eqref{derived_equiv} of) $C^*(ET\times_T X;k)$ in the derived category of dg $S$-modules. When $k=\mathbb{Q}$, the minimal Hirsch--Brown model can be defined as the dg $S$-module minimal model of a Sullivan model for the Borel construction $ET\times_T X$. A convenient reference for the existence, uniqueness and properties of these models is~\cite[Appendix~A]{AZ}.

\begin{corollary}
\label{c_HB_pert}
Let $K$ be a simplicial complex on the vertex set $[m]$ and choose a multigraded deformation retraction for $\mathcal{C}^*_{\mathrm{cw}}(\mathcal{Z}_K;k)$, as in (\ref{e_h_retract_data_Zk}). The cohomology operations on $ H^*(\mathcal{Z}_K;k)$ defined by
\[
\partial_{I} = \sum_{I=\{i_1,...,i_n\}} \pi \iota_{i_1}h \iota_{i_2} h \cdots h \iota_{i_n} \sigma\quad \text{ for } I\subseteq [m]
\]
assemble to yield the minimal Hirsch--Brown model of the action of $T^m$ on $\mathcal{Z}_K$:
\[
\big(S\otimes H^*(\mathcal{Z}_K;k) ,d \big)\xrightarrow{\ \simeq\ } H^*_{T^m}(\mathcal{Z}_K;k)  \quad \text{ with } d = \sum_{I\subseteq [m]} v_{I}\otimes \partial_{I}.
\]
\end{corollary}

\begin{proof}
By~\cref{t_coh_ops_perturb}, $(S\otimes H^*(\mathcal{Z}_K;k),d)$ is quasi-isomorphic to $H^*_{T^m}(\mathcal{Z}_K;k)\cong k[K]$ as a dg $S$-module. It remains to show that $(S\otimes H^*(\mathcal{Z}_K;k),d)$ is quasi-isomorphic to the image of $C^*(ET^m\times_{T^m}\mathcal{Z}_K;k)$ under~\eqref{derived_equiv}, and for this we use a well-known formality result for $ET^m\times_{T^m}\mathcal{Z}_K$. By~\cite[Theorem~4.3.2]{BPbook}, the Borel fibration $ET^m\times_{T^m}\mathcal{Z}_K \to BT^m$ can be identified up to homotopy with the natural inclusion $DJ_K\to BT^m$ of the Davis--Januszkiewicz space\footnote{Here $DJ_K$ is defined analogously to~\eqref{eq_MAC_def} as a polyhedral product corresponding to the pair of spaces $(\mathbb{C}P^\infty,\ast)$ instead of $(D^2,S^1)$.} associated to $K$ into $BT^m=(\mathbb{C}P^\infty)^m$. By~\cite{FRANZ2} or~\cite{NR}, there is a dg algebra quasi-isomorphism $C^*(DJ_K;k)\to H^*(DJ_K;k)$ which is functorial in $K$. Since $DJ_{\Delta^{m-1}}=BT^m$, the inclusion $K\hookrightarrow \Delta^{m-1}$ gives rise to a commutative diagram
\[
\begin{tikzcd}
    C^*(BT^m;k) \ar[r] \ar[d,"\simeq"] & C^*(DJ_K;k) \ar[d,"\simeq"] \\
    H^*(BT^m;k) \ar[r] & H^*(DJ_K;k).
\end{tikzcd}
\]
Since the quasi-isomorphism $f$ defining~\eqref{derived_equiv} can be taken to be the left vertical map in the diagram above, this shows that the $S$-module $H^*_{T^m}(\mathcal{Z}_K;k)\cong H^*(DJ_K;k)$ is mapped under the quasi-inverse of~\eqref{derived_equiv} (restricting scalars along $f$) to a $C^*(BT^m;k)$-module that is quasi-isomorphic to $C^*(ET^m\times_{T^m}\mathcal{Z}_K;k)\cong C^*(DJ_K;k)$. We now know that $(S\otimes H^*(\mathcal{Z}_K;k),d)$ is quasi-isomorphic as a dg $S$-module to $H^*_{T^m}(\mathcal{Z}_K;k)$, which in turn is quasi-isomorphic to the image under~\eqref{derived_equiv} of $C^*(ET^m\times_{T^m}\mathcal{Z}_K;k)$, and this completes the proof.
\end{proof}

The higher operations $\delta_I$ were defined in \cref{s_higher_ops} by Massey-type formulae, up to the indeterminacy explained in \eqref{e_indeterminacy}. The next result shows that these are essentially equal to the operations $\partial_I$ defined in this section. In particular, the indeterminacy of the $\delta_I$ can be removed by fixing a deformation retraction for $\mathcal{C}^*_{\mathrm{cw}}(\mathcal{Z}_K;k)$ as in \eqref{e_h_retract_data_Zk}, after which they assemble to yield the minimal Hirsch--Brown model for the $T^m$-action on $\mathcal{Z}_K$.  

\begin{proposition}\label{p_operations_from_pert}
    The operations $\partial_I$  provide representatives for the higher operations $\delta_I$ where they are defined. More precisely, for an index $I\subseteq [m]$ and an element  $\alpha\in\operatorname{Tor}^S(k[K],k)$, if $\delta_J(\alpha)=0$ for all $J\subsetneq I$, then $\delta_I(\alpha)=\partial_I(\alpha)$ modulo $\sum_{J\subsetneq I}\operatorname{im}(\delta_J)$.    
\end{proposition}

\begin{proof}
    If $I=\{i\}$ then $\partial_i=\delta_i=\iota_i$ by definition, and we proceed inductively. Suppose that $\delta_J(\alpha)$ is defined and zero for all $J\subsetneq I$. For each $i\in I$ we can assume by induction that
    \[
    \alpha_{I\smallsetminus i}= \sum_{I\smallsetminus i=\{i_1,...,i_n\}} \iota_{i_1}h \iota_{i_2} h \cdots h \iota_{i_n} \sigma(\alpha)
    \] 
    is a cycle representing $\delta_{I\smallsetminus i} (\alpha)$. Since we are assuming $\delta_{I\smallsetminus i} (\alpha)=0$, each $\alpha_{I\smallsetminus i}$ is a boundary. Then we choose preimages $d^{-1}\alpha_{I\smallsetminus i}$ under the differential, and by definition of the higher operations set $\delta_I(\alpha)=[\sum_{i\in I} \iota_i d^{-1}\alpha_{I\smallsetminus i}]$. We are free to use $h\alpha_{I\smallsetminus i}=d^{-1}\alpha_{I\smallsetminus i}$ as our preimages, in which case
    \[
    \delta_I(\alpha)=\big[\sum_{i\in I} \iota_i h\alpha_{I\smallsetminus i}\big] = \ \pi\!\!\!\!\sum_{\substack{i\in I,\  I\smallsetminus i= \\ \{i_1,...,i_n\}}} \iota_i h\iota_{i_1}h \iota_{i_2} h \cdots h \iota_{i_n} \sigma(\alpha)  = \sum_{ I= \{i,i_1,...,i_n\}} \!\!\!\pi\iota_i h\iota_{i_1}h \iota_{i_2} h \cdots h \iota_{i_n} \sigma(\alpha).
    \]
    This is the defining formula for $\partial_I(\alpha)$ given in \cref{t_coh_ops_perturb}, so we are done.
\end{proof}

We finish this section by describing Hirsch--Brown models for the actions of subtori $H\subseteq T^m$ (that is, compact connected abelian subgroups) on moment-angle complexes.

Every one-parameter subgroup $S^1\to T^m$ of the $m$-torus is of the form $t\mapsto (t^{a_1},\ldots,t^{a_m})$
for some $a_1,\ldots,a_m\in \mathbb{Z}$. Using this we assign to each subtorus $H\subseteq T^m$ an ideal of $S$ that is generated by linear forms:
\begin{equation}\label{e_J_ideal}
H\ \longmapsto\ \mathcal{J}_H= \bigg( {\textstyle\sum\limits_{i\in [m]}} b_iv_i \;:\; {\textstyle\sum\limits_{i\in [m]}} a_ib_i=0 \text{ for all } S^1\to H,\ t\mapsto (t^{a_1},\ldots,t^{a_m}) \bigg).
\end{equation}
Recalling that the cohomology ring of the classifying space $BT^m=(\mathbb{C}P^\infty)^m$ is given by the polynomial ring $H^*(BT^m)=S=k[v_1,\ldots,v_m]$ with generators in degree $2$, it is straightforward to see that the map $H^*(BT^m)\to H^*(BH)$ induced by the inclusion $H \hookrightarrow T^m$ can be identified with the quotient map $S\to S/{\mathcal{J}_H}$.

\begin{theorem}
\label{t_quotient_HB}
Let $K$ be a simplicial complex on the vertex set $[m]$ and let
\[
\big(S\otimes H^*(\mathcal{Z}_K;k), d \big)\xrightarrow{\ \simeq\ } H^*_{T^m}(\mathcal{Z}_K;k)
\]
be the minimal Hirsch--Brown model for the $T^m$-action on $\mathcal{Z}_K$. Then for any subtorus $H\subseteq T^m$, taking the quotient by $\mathcal{J}_H$ yields the  minimal Hirsch--Brown model for the $H$-action on $\mathcal{Z}_K$:
\[
\big((S/\mathcal{J}_H)\otimes H^*(\mathcal{Z}_K;k),d \big)\xrightarrow{\ \simeq\ } C^*(EH\times_H\mathcal{Z}_K;k).
\]
\end{theorem}

\begin{proof}
All cochain and cohomology groups are taken with coefficients in the field $k$, which we suppress from the notation below. In the morphism of Borel fibrations
\[
\begin{tikzcd}
    \mathcal{Z}_K \ar[d,equals] \ar[r] & EH\times_{H} \mathcal{Z}_K \ar[d] \ar[r] & BH \ar[d] \\
    \mathcal{Z}_K \ar[r] & ET^m\times_{T^m} \mathcal{Z}_K \ar[r] &  BT^m
\end{tikzcd}
\]
the right-hand square is a homotopy pullback. Taking cochains, the Eilenberg--Moore model for the pullback yields a quasi-isomorphism
\begin{equation} \label{e_tensor1}
C^*(EH\times_{H} \mathcal{Z}_K)\simeq C^*(BH) \otimes_{C^*(BT^m)}^{\rm L} C^*(ET^m\times_{T^m}\mathcal{Z}_K)
\end{equation}
in the homotopy category of dg $C^*(BH)$-modules. Since the given Hirsch--Brown model is a semi-free dg $S$-module resolution of $H^*_{T^m}(\mathcal{Z}_K)$, taking the quotient by $\mathcal{J}_H$ computes the derived tensor product, giving a quasi-isomorphism of $H^*(BH)=S/\mathcal{J}_{H}$-modules
\begin{equation} \label{e_tensor2}
\big((S/\mathcal{J}_H)\otimes H^*(\mathcal{Z}_K), d \big)\simeq  S/\mathcal{J}_H\otimes_S^{\rm L}H^*_{T^m}(\mathcal{Z}_K).
\end{equation}
We next construct a quasi-isomorphism between~\eqref{e_tensor1} and~\eqref{e_tensor2}, implying the collapse of the Eilenberg--Moore spectral sequence of the homotopy pullback above. We identify $ET^m\times_{T^m} \mathcal{Z}_K$ with $DJ_K$ as was done in the proof of \cref{c_HB_pert}. As reviewed in~\cite[Section~4.1]{FF}, the construction of the quasi-isomorphism $f\colon C^*(BT^m)\to H^*(BT^m)$ used in the proof of \cref{c_HB_pert} depends on a choice of chain representatives $c_1,\ldots,c_m$ for a basis $x_1,\ldots,x_m$ of $H_1(T^m)$ where each $c_i$ lies in the $i$th coordinate circle factor of $T^m$. Given another basis $x'_1,\ldots,x'_m \in H_1(T^m)$, write $x'_i=\sum_j\alpha_{ij}x_j$ for each $1\leqslant i\leqslant m$, and set $c'_i=\sum_j\alpha_{ij}c_j$. Then it follows from the construction in~\cite[Section~4.1]{FF} that the quasi-isomorphism $f'\colon C^*(BT^m)\to H^*(BT^m)$ corresponding to the chains $c'_1,\ldots,c'_m$ is equal to $f$, while any other set of representatives $\tilde{c}_1,\ldots,\tilde{c}_m$, with $\tilde{c}_i$ homologous to $c'_i$, leads to a quasi-isomorphism $\tilde{f}$ which is homotopic to $f'$ as a map of dg algebras by~\cite[Proposition~4.4]{FF}. By fixing a decomposition into circles for each of the factors of $T^m\cong H\times (T^m/H)$ and choosing a basis of $H_1(T^m)$ represented by chains $\tilde{c}_i$ that lie in the $i$th circle factor for each $1\leqslant i\leqslant m$, we obtain a quasi-isomorphism $\tilde{f}\colon C^*(BT^m)\to H^*(BT^m)$ that is natural with respect to coordinatewise inclusions of subtori by~\cite[Theorem~5.3]{FRANZ2} and is homotopic to $f$. We therefore have a diagram
\[
\begin{tikzcd}
    C^*(BH) \ar[d,"\simeq"] & C^*(BT^m) \ar[d,"f"',"\simeq"] \ar[l] \ar[r] & C^*(DJ_K) \ar[d,"\simeq"] \\
    H^*(BH) & H^*(BT^m) \ar[l] \ar[r] & H^*(DJ_K)
\end{tikzcd}
\]
where the right square strictly commutes, the left square commutes up to a dg algebra homotopy, all vertical maps are quasi-isomorphisms of dg algebras and the horizontal maps are induced by the evident inclusions. This diagram induces the desired quasi-isomorphism between~\eqref{e_tensor1} and~\eqref{e_tensor2}. Moreover, by restricting scalars along the left vertical map, the derived tensor product~\eqref{e_tensor2} becomes a $C^*(BH)$-module quasi-isomorphic to~\eqref{e_tensor1}, so it follows that $\big((S/\mathcal{J}_H)\otimes H^*(\mathcal{Z}_K),d \big)$ is the minimal Hirsch--Brown model for the $H$-action on $\mathcal{Z}_K$. 
\end{proof}

\begin{remark}
A space $X$ with an $H$-action is said to be \emph{MOD-formal} (over $k$) if $C^*(EH\times_H X;k)$ and $H^*_{H}(X;k)$ are isomorphic in the derived category of dg $H^*(BH;k)$-modules; this condition was investigated by Amann and Zoller in connection with the Toral Rank Conjecture \cite{AZ}. The proof of \cref{c_HB_pert} shows that the $T^m$-action on $\mathcal{Z}_K$ is MOD-formal (since $DJ_K\simeq ET^m\times_{T^m}\mathcal{Z}_K$ is a formal space). However this property is not inherited by subtori; see \cite[Example 7.1]{AZ}.
\end{remark}

Given a subset $I\subseteq [m]$, we will write
\begin{equation}\label{e_coord_tori}
    T^I=\big\{(t_1,\ldots,t_m) \,:\, t_i=1\text{ for }i\notin I\big\}\quad\text{and}\quad \mathcal{J}_I=\big(v_i \,:\, i\notin I\big)
\end{equation}
for the \emph{coordinate subtorus} of $T^m$ specified by $I$, and the corresponding ideal of $S$. The next result follows by combining \cref{c_HB_pert} with \cref{t_quotient_HB}.

\begin{corollary} \label{c_coord_HB}
    Let $K$ be a simplicial complex on the vertex set $[m]$ and choose a multigraded deformation retraction for $\mathcal{C}^*_{\mathrm{cw}}(\mathcal{Z}_K;k)$, resulting in cohomology operations $\partial_J$ as in \cref{c_HB_pert}.
    
For any subset $I\subseteq [m]$, the minimal Hirsch--Brown model for the $T^I$-action on $\mathcal{Z}_K$ is given by
\[
\big((S/\mathcal{J}_I)\otimes H^*(\mathcal{Z}_K;k) ,d \big)\xrightarrow{\ \simeq\ } C^*(ET^I\times_{T^I}\mathcal{Z}_K;k)  \quad \text{ with } d = \sum_{J \subseteq I} v_{J}\otimes \partial_{J}.
\]
\end{corollary}

Small models for $H^*_{T^I}(\mathcal{Z}_K)$ are described using Koszul complexes in \cite{PZ} (see also \cite{LMM}). For our purposes, an advantage of the Hirsch--Brown models above is that they make explicit the relationship between the equivariant and ordinary cohomology of $\mathcal{Z}_K$. In particular, the Hirsch--Brown model in \cref{c_coord_HB} (together with \cref{p_operations_from_pert}) shows that the freeness of $H^*_{T^I}(\mathcal{Z}_K)$ over $H^*(BT^I)$ is equivalent to the vanishing of certain cohomology operations on $H^*(\mathcal{Z}_K)$ coming from the $T^I$-action. This will be used in the next two sections to characterise equivariant formality for any (not necessarily coordinate) subtorus action on $\mathcal{Z}_K$.

\section{Equivariant formality and higher operations}

In the section we explain how  equivariant formality of subtorus actions on a moment angle complex can be understood from the minimal free resolution of the corresponding Stanley-Reisner ring, and equivalently from the cohomology operations $\delta_I$. The language of $\mathcal{J}$-closed ideals turns out to be a convenient lens through which to view this relationship.

\subsection{\texorpdfstring{$\mathcal{J}$}{J}-closed modules} \label{s_J_closed}

The class of $\mathcal{J}$-closed modules was introduced in the work of Diethorn on free resolutions in local algebra, as a natural generalisation of the notion of a weak complete intersection ideal ~\cite{DIETHORN}. As before, we write $S=k[v_1,\ldots,v_m]$, multigraded by $\mathbb{N}^m$.

\begin{definition}
Let $\mathcal{J}\subseteq S$ be an ideal and let $M$ be a multigraded $S$-module. Then $M$ is \emph{$\mathcal{J}$-closed} if, in the minimal free resolution 
\[
\cdots\to F_2\xrightarrow{d_2}F_1\xrightarrow{d_1}F_0\to M\to 0,
\]
the differentials satisfy $d_i(F_i)\subseteq \mathcal{J}F_{i-1}$ for all $i$. A monomial ideal $\mathcal{I}\subseteq S$ is called \emph{$\mathcal{J}$-closed} if $S/\mathcal{I}$ is a $\mathcal{J}$-closed module. (In~\cite{DIETHORN} it is assumed that $\mathcal{J}$ is generated by a regular sequence---this will be the case in our main situation of interest, but we do not make that assumption here.)
\end{definition}

Let $\mathcal{I}(d_i)$ denote the ideal generated by the entries of $d_i$, when written as a matrix over $S$ (this is sometimes called the first Fitting ideal of $d_i$, and it is independent of the matrix representing $d_i$). If we also write $\mathcal{I}(M)= \sum_i \mathcal{I}(d_i)$, then by definition $M$ is $\mathcal{J}$-closed if and only if $\mathcal{I}(M)\subseteq \mathcal{J}$.

\begin{lemma}\label{l_Jclosed_multigraded}
Let $\mathcal{J}\subseteq S$ be an ideal (not necessarily multigraded) and let 
\[
\mathcal{J}'=\bigoplus_{a\in \mathbb{N}^m}\mathcal{J}\cap S_a \subseteq S
\]
be the largest multigraded ideal contained in $\mathcal{J}$. Then a multigraded $S$-module $M$ is $\mathcal{J}$-closed if and only if it is $\mathcal{J}'$-closed.
\end{lemma}

\begin{proof}
Since $M$ is multigraded, it has a multigraded minimal resolution, and therefore $\mathcal{I}(M)$ is a multigraded ideal. From this it follows that $\mathcal{I}(M)\subseteq \mathcal{J}$ if and only if $\mathcal{I}(M)\subseteq \mathcal{J}'$.
\end{proof}

Equivariant formality is related to the $\mathcal{J}$-closed condition under the correspondence~\eqref{e_J_ideal}.

\begin{proposition}\label{p_eqformal_Jclosed}
Let $K$ be a simplicial complex on vertex set $[m]$ and let $H\subseteq T^m$ be a subtorus. Then the action of $H$ on $\mathcal{Z}_K$ is equivariantly formal over $k$ if and only if the Stanley--Reisner ring $k[K]$ is $\mathcal{J}_H$-closed as an $S$-module.
\end{proposition}

\begin{proof}
Let $F$ be the minimal free resolution of $k[K]$ over $S$. By \cref{t_quotient_HB} there is an isomorphism $H^*(F/\mathcal{J}_HF)\cong H^*_{T^m}(\mathcal{Z}_K;k)$. Since $F/\mathcal{J}_HF$ is a minimal complex of free $S/\mathcal{J}_H$-modules, the only way it can have free homology is if its differentials vanish entirely. This happens exactly when $d(F)\subseteq \mathcal{J}_HF$, that is, when $k[K]$ is $\mathcal{J}_H$-closed.
\end{proof}

\subsection{Reduction to coordinate subtori} \label{s_reduction}

In this section we show that, for the purpose of answering \cref{question}, it suffices to study the coordinate subtorus actions on moment-angle complexes. For these torus actions we give some characterisations of equivariant formality that will be used and generalised in subsequent sections. 

\begin{definition}
Let $H\subseteq T^m$ be a subtorus. The \emph{coordinate hull} of $H$ is the smallest coordinate subtorus of $T^m$ containing $H$, 
\[
\operatorname{hull}(H)= \Big\{(t_1,\ldots,t_m) \;:\; t_i=1 \text{ if } H\subseteq {\textstyle\prod\limits_{j<i}}S^1\times \{1\} \times {\textstyle\prod\limits_{j>i}}S^1 \subseteq T^m\Big\}.
\]
\end{definition}

In positive characteristic we will need to use the following slightly different (possibly smaller) construction. 

\begin{definition}
Let $H\subseteq T^m$ be a subtorus. For any prime number $p$, the \emph{$p$-coordinate hull} of $H$ is the coordinate subtorus of $T^m$ given by
\[
p\operatorname{-hull}(H)= \big\{(t_1,\ldots,t_m) \;:\; t_i=1 \text{ if } p|a_i\text{ for any map } S^1\to H,\ t\mapsto (t^{a_1},\ldots,t^{a_m})\big\}.
\]
\end{definition}

\begin{proposition}\label{p_coord_torus}
Let $K$ be a simplicial complex on vertex set $[m]$ and let $H\subseteq T^m$ be a subtorus.
\begin{enumerate}[label={\normalfont(\arabic*)}]
\item If $k$ is a field of characteristic zero, then the action of $H$ on $\mathcal{Z}_K$ is equivariantly formal over $k$ if and only if the action of $\operatorname{hull}(H)$ is equivariantly formal over $k$. 
\item If $k$ is a field of characteristic $p>0$, then the action of $H$ on $\mathcal{Z}_K$ is equivariantly formal over $k$ if and only if the action of $p\operatorname{-hull}(G)$ is equivariantly formal over $k$.
\end{enumerate}
\end{proposition}

\begin{proof}
Taking one-parameter-subgroups, the inclusion $H\subseteq T^m$ induces an inclusion of lattices 
\[
L_H=\operatorname{Hom}(S^1,H)\subseteq\operatorname{Hom}(S^1,T^m)=\mathbb{Z}^m.
\]
The smallest coordinate subspace of $\mathbb{Q}^m$ containing $L_H\otimes_\mathbb{Z}\mathbb{Q}$ is exactly $L_{{\rm hull}(H)}\otimes_\mathbb{Z}\mathbb{Q}$, and likewise the smallest coordinate subspace of $\mathbb{F}_p^m$ containing $L_H\otimes_\mathbb{Z}\mathbb{F}_p$ is exactly $L_{p\operatorname{-hull}(H)}\otimes_\mathbb{Z}\mathbb{F}_p$.

The inclusion $L_H\subseteq \mathbb{Z}^m$ is dual to a surjection $k^m\twoheadrightarrow \operatorname{Hom}(L_H,k)$. This copy of $k^m$ can be naturally identified with ${\rm span}_k\{v_1,\ldots,v_m\}$, and we define the subspace
\[
V_H =\ker\big({\rm span}_k\{v_1,\ldots,v_m\}\twoheadrightarrow \operatorname{Hom}(L_H,k)\big).
\]
Unraveling all these definitions, the ideal $\mathcal{J}_H\subseteq S$ defined in \eqref{e_J_ideal} is the ideal generated by this subspace $V_H$. By duality, when $k$ has characteristic zero the largest coordinate subspace contained in $V_H$ is  exactly $V_{\operatorname{hull}(H)}$, and likewise when $k$ has characteristic $p$ the largest coordinate subspace contained in $V_H$ is  exactly $V_{p\operatorname{-hull}(H)}$. Combining these ingredients, the statement now follows from \cref{l_Jclosed_multigraded} and \cref{p_eqformal_Jclosed}.
\end{proof}

According to \cref{p_coord_torus}, the problem of determining which subtori $H\subseteq T^m$ act equivariantly formally on $\mathcal{Z}_K$ can be reduced to the case of the coordinate subtori $H=T^I$, where $I\subseteq [m]$, as in~\eqref{e_coord_tori}. We therefore focus on these actions from now on.

\begin{proposition}
\label{c_eq_form_and_coh_ops}
Let $K$ be a simplicial complex on $[m]$ and let $I\subseteq [m]$. Then the following conditions are equivalent:
\begin{enumerate}[label={\normalfont(\alph*)}]
    \item\label{i_eq_formal} the $T^I$-action on $\mathcal{Z}_K$ is equivariantly formal over $k$;
    \item\label{i_Jclosed} the Stanley--Reisner ring $k[K]$ is $\mathcal{J}_I$-closed;
    \item\label{i_vanishingops} the cohomology operations $\delta_U$ vanish on $H^*(\mathcal{Z}_K;k)$ for all $U\subseteq I$.
\end{enumerate}
\end{proposition}

\begin{proof}
The equivalence of \ref{i_eq_formal} and \ref{i_Jclosed} is contained in \cref{p_eqformal_Jclosed}. By \cref{t_coh_ops_perturb} and \cref{p_operations_from_pert}, the minimal free resolution of $k[K]$ over $S$ is of the form $F=S\otimes H^*(\mathcal{Z}_K)$ with differential $d = \sum_{U\subseteq [m]} v_{U}\otimes \delta_{U}$, where $v_{U}=\prod_{i\in U}v_i$. This means $d(F)\subseteq \mathcal{J}_IF$ if and only if $\delta_{U}=0$ whenever $v_{U}\notin \mathcal{J}_I$. Note that $v_{U}\notin \mathcal{J}_I$ if and only if $U\subseteq I$, and altogether this shows that \ref{i_Jclosed} is equivalent to \ref{i_vanishingops}. (The equivalence of \ref{i_eq_formal} and \ref{i_vanishingops} also follows immediately from \cref{c_coord_HB} and \cref{p_operations_from_pert}.)
\end{proof}

\begin{remark}
The equivalence of condition \ref{i_eq_formal} above with the vanishing of an infinite family of cohomology operations $\delta_s$, $s\in k[v_i : i\in I]$, is due to Goresky, Kottwitz and MacPherson~\cite{GKM}. We remark that the proof of \cref{c_eq_form_and_coh_ops} is quite different from the arguments in~\cite{GKM}, and that in our more restrictive setting the finite family of operations appearing in \ref{i_vanishingops} (indexed by squarefree monomials) form a complete set of obstructions for equivariant formality. 
\end{remark}

Combining \cref{c_eq_form_and_coh_ops} with \cref{comm_diagrams} (or Katth\"an's combinatorial description of the linear part of the minimal free resolution of $k[K]$ \cite{K}) also yields the following characterisation of equivariant formality for coordinate circle actions in terms of the homology of full subcomplexes of $K$. It will be generalised to arbitrary coordinate subtori in \cref{t_main_eq_formal}.

\begin{theorem} \label{1-torus_ef}
Let $K$ be a simplicial complex on vertex set $[m]$ and let $j\in[m]$. Then the following conditions are equivalent:
\begin{enumerate}[label={\normalfont(\alph*)}]
\item the coordinate $S^1_j$-action on $\mathcal{Z}_K$ is equivariantly formal over $k$;
\item the derivation $\iota_j$ is trivial on $H^\ast(\mathcal{Z}_K;k)$;
\item $K_{J\smallsetminus j} \hookrightarrow K_J$ induces the trivial map on $\widetilde{H}^\ast(\; ;k)$ for all $J\subseteq [m]$ with $j\in J$.
\end{enumerate}
\end{theorem}

\begin{remark}
It follows that $H^*(\mathcal{Z}_K)$ is a trivial $\Lambda(\iota_1,\ldots,\iota_m)$-module if and only if the coordinate $S^1_j$-action on $\mathcal{Z}_K$ is equivariantly formal for all $j\in [m]$. This latter condition (the simultaneous equivariant formality of every coordinate circle action) is considered in~\cite{PZ}, where some combinatorial characterisations are given in the case that $K$ is flag or $1$-dimensional. (In the flag case, \cite[Theorem~4.9]{PZ} can readily be recovered from the results of \cref{s_flag_case} below.)
\end{remark}

We emphasise that the equivariant formality of a circle action on a space is typically not equivalent to the vanishing of the primary cohomology operation induced by the action as in \cref{1-torus_ef}; in general, the vanishing of an infinite family of higher operations is also necessary (see~\cite[Section~13]{GKM}). In the case considered above, \cref{l_deg_of_ops} automatically implies that $\delta_{v_j^n}=0$ for all $n>1$, leaving the primary operation $\delta_{v_j}=\iota_j$ as the only possible obstruction. 

Below we give two examples of circle actions on manifolds inducing trivial primary operations where equivariant formality is obstructed by associated higher operations. The first example demonstrates that the vanishing of a primary operation alone is not sufficient for the equivariant formality of (non-coordinate) circle actions on moment-angle complexes.

\begin{example} \label{ex_not_sq_free_1}
Let $K=\partial\Delta^1$ and observe that $\mathcal{Z}_K=D^2\times S^1 \cup S^1\times D^2\cong S^3$, so the $\Lambda(\iota_1,\iota_2)$-module structure on $H^*(\mathcal{Z}_K)$ is trivial for degree reasons. In this case, restricting the standard $T^2$-action to the diagonal circle $S^1_\mathrm{diag} = \{(t,t)\} \subseteq T^2$ yields the Hopf action on $S^3$, and the primary operation induced by this circle action is $\iota_1+\iota_2$. The fundamental class in $H^3(\mathcal{Z}_K)$ is represented by the cocycle $v_1u_2\in R(K)=\bigl( k[v_1,v_2]/(v_1v_2) \otimes \Lambda(u_1,u_2) \bigr)/(v_i^2,v_iu_i)$. Since there is no indeterminacy, the zig-zag
\[ v_1u_2 \xmapsto{\mathmakebox[2.2em]{\iota_1+\iota_2}} v_1 \xleftarrow{\mathmakebox[2.2em]{d}}\mapsfromchar u_1 \xmapsto{\mathmakebox[2.2em]{\iota_1+\iota_2}} 1 \]
shows that the secondary operation $H^3(\mathcal{Z}_K)\to H^0(\mathcal{Z}_K)$ induced by the $S^1_\mathrm{diag}$-action maps $[v_1u_2]$ to $[1]$, obstructing equivariant formality.

More generally, for any $m\geqslant 2$, if $K=\partial\Delta^{m-1}$ on the vertex set $[m]$, then the $m$-ary operation induced by the diagonal circle action on $\mathcal{Z}_K\cong S^{2m-1}$ is given by the higher operation $\delta_{[m]}$ and defines an isomorphism $H^{2m-1}(\mathcal{Z}_K)\to H^0(\mathcal{Z}_K)$. 
\end{example}

\begin{example} \label{ex_not_sq_free_2}
Consider the real solvable Lie algebra 
\[
L=\langle W,X,Y,Z \;:\; [W,X]=X,\, [W,Y]=-Y,\, [X,Y]=Z \rangle.
\]
The corresponding simply connected solvable Lie group $G$ admits a lattice $\Gamma$, and the de Rham cohomology of the solvmanifold $G/\Gamma$ is computed by the Lie algebra cohomology of $L$ since the inclusion of left-invariant differential forms $\Lambda L^\vee=\Omega^*_G(G) \hookrightarrow \Omega^*_\Gamma(G)=\Omega^*(G/\Gamma)$ is a quasi-isomorphism by a theorem of Hattori \cite{HATTORI}. Let $(\Lambda L^\vee,d)=(\Lambda(w,x,y,z),d)$ be the Chevalley--Eilenberg complex of $L$ with $|w|=|x|=|y|=|z|=1$ and differential determined by
\[ d(x)=wx, \quad d(y)=-wy, \quad d(z)=xy.\]

It can be shown that $S^1$ acts smoothly on $G/\Gamma$ with fundamental vector field $Z$ (see for example \cite[Theorem~3.6]{BOCK}). The corresponding cohomology operation is induced by the degree $-1$ derivation of the de Rham complex given by interior multiplication by $Z$ (cf.~\cite[Section~10.5]{GKM}). Let $\iota_Z\colon \Lambda^*L^\vee \to \Lambda^{*-1}L^\vee$ denote the derivation defined by interior multiplication by $Z$. Then a straightforward computation with the dg $\Lambda(\iota_Z)$-module $(\Lambda(w,x,y,z),d)$ shows that $H^*(G/\Gamma;\mathbb{R})$ together with all higher operations induced by the $S^1$-action is given by:
\[ \hspace{-1cm}
\begin{tikzcd}[row sep=small]
H^4\colon & \left[wxyz\right] \ar[ddd,swap,"-\delta_{Z^2}",bend right=50] \\
H^3\colon & \left[xyz\right] \ar[ddd,"\delta_{Z^2}",bend left=50] \\
 & \\
H^1\colon & \left[w\right] \\
H^0\colon & \left[1\right]
\end{tikzcd}
\]
In particular, the primary operation $\delta_Z=\iota_Z$ acts trivially on $H^*(G/\Gamma;\mathbb{R})$, but the $S^1$-action is not equivariantly formal since the induced secondary operation $\delta_{Z^2}$ is nontrivial. 
\end{example}

\subsection{Equivariant formality in the flag case} \label{s_flag_case} 

For flag complexes we already have enough machinery to completely characterise the subtori that act equivariantly formally on $\mathcal{Z}_K$ in terms of the combinatorics of $K$. We find that in this case the answer does not depend on the characteristic of the field involved. 

A simplicial complex $K$ is called a \emph{flag complex} if every set of vertices of $K$ which are pairwise connected by edges spans a simplex in $K$. Equivalently, $K$ is flag if its corresponding Stanley--Reisner ideal is quadratic. 

The \emph{join} of two simplicial complexes $K$ and $L$ on disjoint vertex sets is defined to be the simplicial complex 
\[ K\ast L=\{\sigma\cup\tau \;:\; \sigma\in K,\ \tau\in L\} \]
on the union of the vertex sets of $K$ and $L$. In the case that $L=\{v\}$ is a single vertex, the join $K\ast \{v\}$ is called the \emph{cone} over $K$.

\begin{lemma} \label{l_flagcone}
Let $K$ be a flag complex on vertex set $[m]$ and let $v\in [m]$. If $\{i,v\} \in K$ for every vertex $i\in [m]$, then $K=K_{[m]\smallsetminus v} \ast \{v\}$.   
\end{lemma}

\begin{proof}
Since $K$ is flag, so is every full subcomplex of $K$. It follows that the join of full subcomplexes $K_{[m]\smallsetminus v} \ast \{v\}$ is flag. By definition, a flag complex is completely determined by its $1$-skeleton, so the result follows from the fact that $K$ and $K_{[m]\smallsetminus v} \ast \{v\}$ have the same $1$-skeleton by the assumption that $v$ is connected by an edge to every vertex of $K$. 
\end{proof}

\begin{lemma} \label{l_flagchar}
Let $K$ be a flag complex on vertex set $[m]$ and let $v\in [m]$. Then the following conditions are equivalent:
\begin{enumerate}[label={\normalfont(\alph*)}]
\item\label{i_flag_lem_a} $K_{J\smallsetminus v} \hookrightarrow K_J$ induces the trivial map on $\widetilde{H}^\ast(\; ;k)$ for all $J\subseteq [m]$ with $v\in J$;
\item\label{i_flag_lem_b} $K_{\{i,j\}} \ast \{v\} \subseteq K$ for every missing edge $\{i,j\} \notin K$ with $v\notin \{i,j\}$.
\end{enumerate}
\end{lemma}

\begin{proof}
That condition \ref{i_flag_lem_a} implies condition \ref{i_flag_lem_b} is clear since for any missing edge $\{i,j\} \notin K$, the inclusion $K_{\{i,j\}} \hookrightarrow K_{\{i,j,v\}}$ is nontrivial in reduced cohomology unless the  vertex $v$ cones over $K_{\{i,j\}}=\partial\Delta^1$.

Conversely, suppose $K_{J\smallsetminus v} \hookrightarrow K_J$ is nontrivial on $\widetilde{H}^\ast(\;)$ for some $J\subseteq [m]$, $v\in J$. Then there must exist a vertex $i\in J$ with $\{i,v\}\notin K$ since otherwise $K_J$ is a cone by \cref{l_flagcone} and hence contractible, contradicting the assumption. Similarly, there must exist some $j\in J\smallsetminus v$ with $\{i,j\}\notin K$ since otherwise $K_{J\smallsetminus v}$ is the cone $K_{J\smallsetminus \{v,i\}}\ast \{i\}$ by \cref{l_flagcone}. Now since $\{i,j\}\notin K$ and $\{i,v\}\notin K$, we have $K_{\{i,j\}}\ast \{v\} \not\subseteq K$.
\end{proof}

While condition \ref{i_flag_lem_a} above apparently involves the coefficient ring $k$, condition \ref{i_flag_lem_b} does not. As we will see next, this will mean that in the flag case the equivariant fomality of a coordinate torus action is independent of $k$ as well. This is not true in general, see \cref{s_dependence_on_char} for examples.

\begin{theorem} \label{t_flag_ef}
Let $K$ be a flag complex on vertex set $[m]$ and let $I\subseteq [m]$. The coordinate $T^I$-action on $\mathcal{Z}_K$ is equivariantly formal (over $k$) if and only if $I\in K$ and $K_{\{i,j\}}\ast K_{I\smallsetminus \{i,j\}} \subseteq K$ for every missing edge $\{i,j\} \notin K$.
\end{theorem}

\begin{proof}
Assume that the $T^I$-action on $\mathcal{Z}_K$ is equivariantly formal. Then $\delta_U=0$ for all $U\subseteq I$ by \cref{c_eq_form_and_coh_ops}. In particular, the secondary operation $\delta_{ij}$ is trivial on $H^3(\mathcal{Z}_K)$ for all $i,j\in I$. Since $H^3(\mathcal{Z}_K)\cong H^3(R(K))$ is spanned by $\{[v_iu_j] \,:\, \{i,j\}\notin K\}$ and $\delta_{ij}[v_iu_j]=1$ for each missing edge $\{i,j\}\notin K$, it follows that $K_I$ has no missing edges. Since $K_I$ is flag, this implies $K_I$ is a simplex, or equivalently $I\in K$. Now to show that $K_{\{i,j\}}\ast K_{I\smallsetminus \{i,j\}} \subseteq K$ for every $\{i,j\} \notin K$, it suffices by flagness to show $K_{\{i,j\}}\ast \{v\} \subseteq K$ for all $v\in I\smallsetminus \{i,j\}$. Since the primary operation $\delta_v=\iota_v$ is trivial for all $v\in I$ by the assumption, this follows from \cref{1-torus_ef} and \cref{l_flagchar}.

Conversely, suppose that $I\in K$ and $K_{\{i,j\}}\ast K_{I\smallsetminus \{i,j\}} \subseteq K$ for every $\{i,j\} \notin K$. Assume toward contradiction that the $T^I$-action on $\mathcal{Z}_K$ is not equivariantly formal, and hence that $\delta_U\neq 0$ for some $U\subseteq I$. Assuming without loss of generality that $U$ is a minimal such subset of $I$, it follows from \cref{l_deg_of_ops} that $\delta_U$ acts nontrivially on some multidegree $J$ part of $H^*(\mathcal{Z}_K)$ with $U\subseteq J$, defining a nontrivial map
\[
\delta_U\colon \widetilde{H}^p(K_J) \longrightarrow \widetilde{H}^{p-|U|+1}(K_{J\smallsetminus U}).
\]
In particular, neither $K_J$ nor $K_{J\smallsetminus U}$ are contractible. This implies that every $j\in J\smallsetminus U$ is contained in a missing edge in $K_{J\smallsetminus U}$, since otherwise $K_{J\smallsetminus U}$ is a cone by \cref{l_flagcone} and is thus contractible. But every $i\in U$ cones over every missing edge in $J\smallsetminus U$ by the assumption. Therefore every $i\in U$ is connected by an edge to every vertex $j\in J\smallsetminus U$. Since every $i\in U$ is also connected by an edge to every other vertex in $U$ (as $U\subseteq I\in K$), it follows that $K_J$ is a cone by \cref{l_flagcone}, a contradiction.
\end{proof}

\begin{remark}
For any simplicial complex $K$, the equivariant formality of the coordinate $T^I$-action on $\mathcal{Z}_K$ implies that $I\in K$. As in the proof above, this can be seen by observing that a missing face in $K_I$ would imply that $1\in H^0(\mathcal{Z}_K)$ is in the image of some $\delta_U$ with $U\subseteq I$. This also follows immediately from the well-known fact that equivariantly formal actions have fixed points. (Note that the $T^I$-action has fixed points precisely when $I\in K$ by definition~\eqref{eq_MAC_def} of $\mathcal{Z}_K$).
\end{remark}

An interesting consequence of \cref{t_flag_ef} is that in the flag case the equivariant formality of a subtorus action on $\mathcal{Z}_K$ is completely determined by the action of primary and secondary cohomology operations on the groups $H^4(\mathcal{Z}_K)$ and $H^3(\mathcal{Z}_K)$, respectively. 

\begin{corollary} \label{c_flag_low_ops}
Let $K$ be a flag complex on vertex set $[m]$ and let $I\subseteq [m]$. If
\[
\delta_v\colon H^4(\mathcal{Z}_K)\to H^3(\mathcal{Z}_K) \quad\text{ and }\quad \delta_{ij}\colon H^3(\mathcal{Z}_K)\to H^0(\mathcal{Z}_K)
\]
are trivial for all $i,j,v\in I$, then $\delta_U$ vanishes everywhere for all $U\subseteq I$. In particular, a coordinate $S^1_v$-action on $\mathcal{Z}_K$ is equivariantly formal if and only if $\delta_v$ is trivial on $H^4(\mathcal{Z}_K)$.
\end{corollary}

\begin{proof}
It is straightforward to check that a secondary operation $\delta_{ij}$ is nontrivial on $H^3(\mathcal{Z}_K)$ if and only if $\{i,j\}\notin K$. Therefore if $\delta_{ij}$ is trivial on $H^3(\mathcal{Z}_K)$ for all $i,j\in I$, then $K_I$ is a simplex by flagness, so $I\in K$. If additionally $\delta_v$ is trivial on $H^4(\mathcal{Z}_K)$ for all $v\in I$, then $K_{\{i,j\}}\hookrightarrow K_{\{i,j,v\}}$ is trivial in reduced cohomology for every $\{i,j\}\notin K$ with $v\in I\smallsetminus \{i,j\}$ by \cref{comm_diagrams}. This implies $K_{\{i,j\}}\ast \{v\}\subseteq K$ for every $\{i,j\}\notin K$ with $v\in I\smallsetminus \{i,j\}$. It follows that both combinatorial conditions of \cref{t_flag_ef} are satisfied, so the $T^I$-action on $\mathcal{Z}_K$ is equivariantly formal.
\end{proof}

The equivariant formality of the action of any subtorus $H\subseteq T^m$ on a moment-angle complex $\mathcal{Z}_K$ can be read off from the minimal free resolution of the Stanley--Reisner ring $k[K]$ by \cref{p_eqformal_Jclosed}. If $K$ is flag, it follows from the above that the equivariant formality of the $H$-action can be read off from the first two differentials in the resolution alone.

\subsection{\texorpdfstring{$\mathcal{J}$}{J}-closed edge ideals}

We obtain as another consequence of \cref{t_flag_ef} a simple classification of $\mathcal{J}$-closed edge ideals for any ideal $\mathcal{J}\subseteq S=k[v_1,\ldots,v_m]$ generated by linear forms.

Let $G$ be a simple graph with vertex set $[m]$ and edge set $E(G)$. The \emph{edge ideal} of $G$ is the quadratic monomial ideal
\[ \mathcal{I}(G)=\big(v_iv_j \,:\, \{i,j\}\in E(G)\big) \subseteq S. \]
The edge ideal of $G$ is the Stanley--Reisner ideal of a simplicial complex associated to $G$ called the \emph{independence complex} ${\operatorname{Ind}(G)}$, defined by
\[ \operatorname{Ind}(G)=\{ \sigma\subseteq [m] \,:\, \sigma \text{ is an independent set of } G \}. \]
Equivalently, $\operatorname{Ind}(G)$ is the unique flag complex with $1$-skeleton the graph complement~$G^c$.

Next, let $\mathcal{J}\subseteq S$ be an ideal generated by linear forms. To classify all graphs $G$ whose edge ideals are $\mathcal{J}$-closed, it suffices by \cref{p_eqformal_Jclosed} and \cref{p_coord_torus} to assume that $\mathcal{J}$ has the form $\mathcal{J}_I=(v_i : i\notin I)$ for some subset $I\subseteq [m]$. 

\begin{corollary}
Let $G$ be a simple graph on vertex set $[m]$ and let $I\subseteq [m]$. Then the following conditions are equivalent:
\begin{enumerate}[label={\normalfont(\alph*)}]
    \item the edge ideal $\mathcal{I}(G)$ is $\mathcal{J}_I$-closed; \label{i_edge_lem_a}
    \item $d_2(F_2)\subseteq \mathcal{J}_IF_1$ and $d_1(F_1)\subseteq \mathcal{J}_IF_0$ in the minimal free resolution $(F,d)$ of $S/\mathcal{I}(G)$; \label{i_edge_lem_b}
    \item $I$ is an independent set of $G$ and $\{i,v\}, \{j,v\} \notin E(G)$ for every edge $\{i,j\}\in E(G)$ and $v\in I\smallsetminus \{i,j\}$. \label{i_edge_lem_c}
\end{enumerate}
\end{corollary}

\begin{proof}
Since $S/\mathcal{I}(G)$ is the Stanley--Reisner ring $k[K]$ of the independence complex $K=\operatorname{Ind}(G)$ and since $S\otimes H^4(\mathcal{Z}_K)$ and $S\otimes H^3(\mathcal{Z}_K)$ lie in homological degree $2$ and $1$, respectively, in the minimal free resolution $F=S\otimes H^*(\mathcal{Z}_K)$ (see \cref{t_coh_ops_perturb} and \cref{p_operations_from_pert}), it follows from condition \ref{i_edge_lem_b} that $\delta_v\colon H^4(\mathcal{Z}_K)\to H^3(\mathcal{Z}_K)$ and $\delta_{ij}\colon H^3(\mathcal{Z}_K)\to H^0(\mathcal{Z}_K)$ are trivial for $i,j,v\in I$. By \cref{c_flag_low_ops}, this implies $\delta_U=0$ for all $U\subseteq I$, so $I(G)$ is $\mathcal{J}_I$-closed by \cref{c_eq_form_and_coh_ops}. This shows \ref{i_edge_lem_b} implies \ref{i_edge_lem_a}. Since \ref{i_edge_lem_a} implies \ref{i_edge_lem_b} by definition, \ref{i_edge_lem_a} and \ref{i_edge_lem_b} are equivalent.   

To see that \ref{i_edge_lem_a} and \ref{i_edge_lem_c} are equivalent, first observe that \ref{i_edge_lem_a} is equivalent to the equivariant formality of the $T^I$-action on $\mathcal{Z}_K$ by \cref{c_eq_form_and_coh_ops}, where $K=\operatorname{Ind}(G)$. Since $K$ is flag, this in turn is equivalent to the condition that $I\in K$ and $K_{\{i,j\}}\ast \{v\} \subseteq K$ for every $\{i,j\}\notin K$ with $v\in I\smallsetminus \{i,j\}$ by \cref{t_flag_ef}. These combinatorial conditions translate precisely to condition \ref{i_edge_lem_c} by definition of the the independence complex of $G$.
\end{proof}

\subsection{Dependence of equivariant formality on the characteristic} \label{s_dependence_on_char}

In contrast with the flag case, we give here some examples of torus actions on moment-angle complexes where the equivariant formality of the action depends on the characteristic of $k$. These examples show that a characterisation of equivariant formality for $\mathcal{Z}_K$ purely in terms of the combinatorics of $K$ (as in \cref{t_flag_ef}) cannot be expected in general.

\begin{example}
Let $K=\partial\Delta^1$ so that $\mathcal{Z}_K\cong S^3$ (as in \cref{ex_not_sq_free_1}), and consider the subgroup $S^1 \to T^2$, $t\mapsto (t,t^p)$, where $p$ is prime. The primary operation $\iota_1+p\iota_2$ induced by this circle action is trivial in cohomology, but the cochain-level zig-zag
\[ v_1u_2 \xmapsto{\mathmakebox[2.3em]{\iota_1+p\iota_2}} pv_1 \xleftarrow{\mathmakebox[2.3em]{d}}\mapsfromchar pu_1 \xmapsto{\mathmakebox[2.3em]{\iota_1+p\iota_2}} p \]
in the reduced Koszul complex $R(K)$ shows that the associated secondary operation maps the fundamental class $[v_1u_2]\in H^3(\mathcal{Z}_K;k)$ to $[p]\in H^0(\mathcal{Z}_K;k)$. Since all tertiary and higher operations are trivial for degree reasons, it follows that this circle action is equivariantly formal over $k=\mathbb{F}_p$ but not over $k=\mathbb{Q}$.
\end{example}

\begin{remark}
The example above illustrates the necessity of the $p$-coordinate hull construction used in \cref{p_coord_torus}. If $H=S^1_{(1,2)}\subseteq T^2$ denotes the subtorus described above, then the coordinate hull of $H$ is $\operatorname{hull}(H)=T^2$, which does not act equivariantly formally (over any $k$) since the Stanley--Reisner ring $H^*_{T^2}(\mathcal{Z}_K;k)\cong k[v_1,v_2]/(v_1v_2)$ of $K=\partial\Delta^1$ is not a free module over $H^*(BT^2;k)=k[v_1,v_2]$. Thus, when working over $k=\mathbb{F}_p$, the equivariant formality of the $H$-action is not detected by the smallest coordinate subtorus $\operatorname{hull}(H)$ containing $H$. In this example the $p$-coordinate hull is $p\operatorname{-hull}(H)=S^1_1$, the first coordinate circle in $T^2$, which does act equivariantly formally on $\mathcal{Z}_K\cong S^3$. 
\end{remark}

The minimal $6$-vertex triangulation of $\mathbb{R}P^2$ is the smallest simplicial complex $K$ for which the associated moment-angle complex $\mathcal{Z}_K$ has $2$-torsion in integral cohomology. (The homotopy type of this moment-angle complex is worked out in~\cite{GPTW}.) Using \cref{comm_diagrams}, it is straightforward to check that $\iota_j\colon H^6(\mathcal{Z}_K;k)\to H^5(\mathcal{Z}_K;k)$ is nonzero for all $j=1,\ldots,6$ and all fields $k$ in this case. It follows that no coordinate $T^I$-action on $\mathcal{Z}_K$ is equivariantly formal over any field $k$. 

The complex $\hat{K}$ in the following example is obtained from the $6$-vertex triangulation $K$ of $\mathbb{R}P^2$ by introducing a seventh vertex which cones over all~$10$ minimal non-faces of $K$. This introduces~$10$ new minimal non-faces containing the seventh vertex and has the effect that the primary cohomology operation $\iota_7$ acts trivially on all cohomology groups except for $H^{10}(\mathcal{Z}_{\hat{K}};k)$ when $\operatorname{char}(k)=2$. According to the commutative diagrams of \cref{comm_diagrams}, $\iota_7\colon H^{10}(\mathcal{Z}_{\hat{K}};k)\to H^{9}(\mathcal{Z}_{\hat{K}};k)$ is determined by the map 
\[
\widetilde{H}^2(\hat{K};k) \longrightarrow \widetilde{H}^2(\hat{K}_{\{1,\ldots,6\}};k)=\widetilde{H}^2(\mathbb{R}P^2;k)
\]
induced by the inclusion of full subcomplexes $\hat{K}_{\{1,\ldots,6\}} \hookrightarrow \hat{K}_{\{1,\ldots,7\}}=\hat{K}$.

\begin{example}
Consider the simplicial complex $\hat{K}$ on $7$ vertices with Stanley--Reisner ideal
\begin{align*}
\big(v_{124}, v_{126}, &v_{134}, v_{135}, v_{156}, v_{235}, v_{236}, v_{245}, v_{346}, v_{456},\\ &v_{1237}, v_{1257}, v_{1367}, v_{1457}, v_{1467}, v_{2347}, v_{2467}, v_{2567}, v_{3457}, v_{3567}\big),
\end{align*}
in $S=k[v_1,\ldots,v_7]$, where $v_{i_1\cdots i_q}=v_{i_1}\!\cdots v_{i_q}$. One can verify with {\tt Macaulay2} \cite{GS} that $v_7$ appears as an entry in the matrix representing the final differential in the minimal free resolution of the Stanley--Reisner ring $\mathbb{F}_2[\hat{K}]$, whereas $v_7$ does not appear as an entry in the differentials of the minimal free resolution of $\mathbb{Q}[\hat{K}]$. In other words, for $\mathcal{J}=(v_1,\ldots,v_6)$, the Stanley--Reisner ring $k[\hat{K}]$ is $\mathcal{J}$-closed when $k=\mathbb{Q}$ but not when $k=\mathbb{F}_2$. Consequently, by \cref{c_eq_form_and_coh_ops}, the coordinate $S^1_7$-action on $\mathcal{Z}_{\hat{K}}$ is equivariantly formal over $\mathbb{Q}$ but not over $\mathbb{F}_2$.
\end{example}

\section{Combinatorial models for the higher operations} \label{comb_models_sec}

In this section we describe the action of the higher cohomology operations for moment-angle complexes in terms of the Hochster decomposition $H^\ast(\mathcal{Z}_K)\cong \bigoplus_{U\subseteq [m]} \widetilde{H}^*(K_U)$. This has already been done for the primary operations $\delta_i=\iota_i$ (see \cref{comm_diagrams}). In general, an analogous description of the higher operations $\delta_U$, for $U\subseteq [m]$, purely in terms of the combinatorics of full subcomplexes of $K$ (and avoiding the issue of indeterminacy) is only possible when the lower degree operations $\delta_V$ all vanish, for $V\subsetneq U$. This will lead in \cref{s_equivariant_formality_general} to a characterisation of equivariantly formal torus actions in terms of the combinatorics of subcomplexes of $K$, generalising \cref{1-torus_ef}.

\subsection{Secondary operations and the Mayer--Vietoris sequence} \label{s_MV_secondary}
We begin by describing the action of the secondary cohomology operations $\delta_{ij}$, since these admit a particularly simple description in terms of the Mayer--Vietoris long exact sequence. However, this will be generalised to all cohomology operations in \cref{s_MV_ss_description}, and that section does not rely on the results of this one.

Fix a multidegree $J\subseteq [m]$ and assume that $i,j\in J$ with $i<j$. 
Consider the subcomplex $K_{J\smallsetminus i}\cup K_{J\smallsetminus j}$ of $K_J$. Since $K_{J\smallsetminus i}\cap K_{J\smallsetminus j}=K_{J\smallsetminus ij}$, the cover $\{K_{J\smallsetminus i},K_{J\smallsetminus j}\}$ of $K_{J\smallsetminus i}\cup K_{J\smallsetminus j}$ gives rise to a Mayer--Vietoris sequence
\[
\begin{tikzcd}
\cdots \ar[r] &  \widetilde{H}^*(K_{J\smallsetminus i}\cup K_{J\smallsetminus j}) \ar[r, "\binom{-\mathrm{res}_i}{\mathrm{res}_j}"] &[1.3em] \widetilde{H}^*(K_{J\smallsetminus i})\oplus \widetilde{H}^*(K_{J\smallsetminus j})\ar[r, "\left(\mathrm{res}_j\;\mathrm{res}_i\right)\;"] &[1.3em] \widetilde{H}^*(K_{J\smallsetminus ij})\ar[r] & \cdots
\end{tikzcd}
\]
where $\operatorname{res}_i$ and $\mathrm{res}_j$ are restriction maps induced by the evident inclusions. We will often simply write $\alpha|_i$ for $\mathrm{res}_i(\alpha)$. Recall that on the summand $\widetilde{H}^*(K_J)$ of $H^*(\mathcal{Z}_K)\cong \bigoplus_{U\subseteq [m]} \widetilde{H}^*(K_U)$, the primary operations
\[ \iota_i\colon \widetilde{H}^*(K_J) \to \widetilde{H}^*(K_{J\smallsetminus i}) \quad\text{ and }\quad \iota_j\colon \widetilde{H}^*(K_J) \to \widetilde{H}^*(K_{J\smallsetminus j}) \]
are equal to the restriction maps up to a sign (see Lemma~\ref{comm_diagrams} and Remark~\ref{mod_structures}). 

If both $\iota_i$ and $\iota_j$ are zero on $\bigoplus_{U\subseteq [m]} \widetilde{H}^*(K_U)$, then this induces a \emph{unique} lift $\ell$ through the connecting homomorphism of the Mayer--Vietoris sequence:
\begin{equation} \label{MV_diagram}
\begin{tikzcd}[column sep=7.5mm]
 & & \widetilde{H}^\ast(K_{J}) \ar[d] \ar[dr,"0"] \ar[dl, dashed, "\ell"'] \\
\cdots \ar[r,"0"] & \widetilde{H}^{\ast-1}(K_{J\smallsetminus ij}) \ar[r] &  \widetilde{H}^\ast(K_{J\smallsetminus i}\cup K_{J\smallsetminus j}) \ar[r] & \widetilde{H}^\ast(K_{J\smallsetminus i})\oplus \widetilde{H}^\ast(K_{J\smallsetminus j})\ar[r] & \cdots \
\end{tikzcd}
\end{equation}

\begin{proposition}\label{MV_lemma}
Under the assumptions above, $\delta_{ij}= (-1)^{\varepsilon(ij,J)}\ell$ on $\widetilde{H}^\ast(K_{J})$.
\end{proposition}

\begin{proof}
Let $[\alpha] \in \widetilde{H}^\ast(K_J)$, and write $[\bar{\alpha}] \in \widetilde{H}^\ast(K_{J\smallsetminus i}\cup K_{J\smallsetminus j})$ for the restriction of $[\alpha]$ along the inclusion $K_{J\smallsetminus i}\cup K_{J\smallsetminus j} \hookrightarrow K_J$. Then $\ell([\alpha])$ is the unique preimage of $[\bar{\alpha}]$ under the connecting homomorphism $\widetilde{H}^{\ast-1}(K_{J\smallsetminus ij}) \to \widetilde{H}^\ast(K_{J\smallsetminus i}\cup K_{J\smallsetminus j})$. Consider the commutative diagram
\[
\begin{tikzcd}
\widetilde{C}^{\ast-1}(K_{J\smallsetminus i} \cup K_{J\smallsetminus j}) \ar[r] \ar[d, "d"] & \widetilde{C}^{\ast-1}(K_{J\smallsetminus i}) \oplus \widetilde{C}^{\ast-1}(K_{J\smallsetminus j}) \ar[r] \ar[d, "d"] & \widetilde{C}^{\ast-1}(K_{J\smallsetminus ij}) \ar[d, "d"] \\
\widetilde{C}^\ast(K_{J\smallsetminus i} \cup K_{J\smallsetminus j}) \ar[r] & \widetilde{C}^\ast(K_{J\smallsetminus i}) \oplus \widetilde{C}^\ast(K_{J\smallsetminus j}) \ar[r] & \widetilde{C}^\ast(K_{J\smallsetminus ij}). \ 
\end{tikzcd}
\]
According to the snake lemma, $\ell([\alpha])$ is represented by a cochain
\begin{equation} \label{eqline1}
-\left(d^{-1}(\alpha|_i)\right)|_j + \left(d^{-1}(\alpha|_j)\right)|_i \in \widetilde{C}^{\ast-1}(K_{J\smallsetminus ij})
\end{equation}
obtained by pushing $\bar{\alpha}$ through the zig-zag from the bottom left to top right corner of the diagram above. Here we are writing $d^{-1}(\alpha|_i)$ for a choice of preimage of $\alpha|_i \in \widetilde{C}^\ast(K_{J\smallsetminus i})$ under the differential, and similarly for $d^{-1}(\alpha|_j)$.

On the other hand, $\delta_{ij}[\alpha] = [\iota_i\beta_j+\iota_j\beta_i]$ for any cochains $\beta_i, \beta_j$ satisfying $\iota_i\alpha=d\beta_i$ and $\iota_j\alpha=d\beta_j$. Since
\[ \iota_i\alpha=(-1)^{\varepsilon(i,J)+|\alpha|+1}\alpha|_i \quad\text{ and }\quad \iota_j\alpha=(-1)^{\varepsilon(j,J)+|\alpha|+1}\alpha|_j \]
by definition of the $\Lambda$-module structure on $\bigoplus_{J\subseteq [m]} \widetilde{C}^*(K_J)$ (see Remark~\ref{mod_structures}), it follows that
\begin{align}
\delta_{ij}[\alpha] &= \left[ (-1)^{\varepsilon(j,J)+|\alpha|+1}\iota_i\left(d^{-1}(\alpha|_j)\right) + (-1)^{\varepsilon(i,J)+|\alpha|+1}\iota_j\left(d^{-1}(\alpha|_i)\right) \right]  \nonumber \\
&= \Big[ (-1)^{\varepsilon(j,J)+|\alpha|+1}(-1)^{\varepsilon(i,J\smallsetminus j)+|\alpha|} \left(d^{-1}(\alpha|_j)\right)|_i \nonumber \\
&\qquad\qquad + (-1)^{\varepsilon(i,J)+|\alpha|+1}(-1)^{\varepsilon(j,J\smallsetminus i)+|\alpha|} \left(d^{-1}(\alpha|_i)\right)|_j \Big] \nonumber \\
&=\Big[ (-1)^{\varepsilon(j,J)+\varepsilon(i,J\smallsetminus j)+1} \left(d^{-1}(\alpha|_j)\right)|_i + (-1)^{\varepsilon(i,J)+\varepsilon(j,J\smallsetminus i)+1} \left(d^{-1}(\alpha|_i)\right)|_j \Big]. \label{eqline2}
\end{align}
Observe that since $i<j$,  $\varepsilon(i,J\smallsetminus j)=\varepsilon(i,J)$ while $\varepsilon(j,J\smallsetminus i)=\varepsilon(j,J)-1$. Therefore the two terms in \eqref{eqline2} have opposite signs. Comparing with \eqref{eqline1}, we conclude that $\ell([\alpha])$ and $\delta_{ij}[\alpha]$ are equal up to the indicated sign.
\end{proof}

\begin{remark}
The combinatorial interpretation of the secondary operation $\delta_{ij}[\alpha]$ given above holds just as well under the weaker assumption that $\iota_i$ and $\iota_j$ vanish only on the summand $\widetilde{H}^{|\alpha|}(K_J)$ containing $[\alpha]$ and on $\widetilde{H}^{|\alpha|-1}(K_{J\smallsetminus i})$ and $\widetilde{H}^{|\alpha|-1}(K_{J\smallsetminus j})$.
More generally, for any cohomology class $[\alpha] \in \widetilde{H}^\ast(K_J)$ in the kernel of the primary operations $\iota_i$ and $\iota_j$, the secondary operation $\delta_{ij}[\alpha]$ is defined up to indeterminacy analogous to the indeterminacy of a triple Massey product. In this case, a statement analogous to \cref{MV_lemma} still holds, with indeterminacy corresponding to the nonuniqueness of a choice of lift $\ell$ in diagram~\eqref{MV_diagram}.  
\end{remark}

It is well known that the connecting homomorphism in the Mayer--Vietoris sequence for an excisive triad $(X;U,V)$ is induced by a map of spaces, namely, the quotient map collapsing the ends of the double mapping cylinder $U\cup \left((U\cap V)\times [0,1]\right)\cup V \simeq X$ to form $\Sigma(U\cap V)$. For the Mayer--Vietoris sequence above, we can see the connecting homomorphism $\widetilde{H}^{\ast-1}(K_{J\smallsetminus ij}) \to \widetilde{H}^\ast(K_{J\smallsetminus i}\cup K_{J\smallsetminus j})$ even more concretely,  being induced by the inclusion $K_{J\smallsetminus i}\cup K_{J\smallsetminus j} \hookrightarrow \Sigma K_{J\smallsetminus ij}$, where $\Sigma K_{J\smallsetminus ij}$ is viewed as the union of the cones $K_{J\smallsetminus ij} \ast\{i\}$ and $K_{J\smallsetminus ij} \ast\{j\}$ (cf.\ Figure \ref{fig1}).

Thus, just as the primary operations $\delta_i$ on $H^*(\mathcal{Z}_K)$ are determined by the maps $K_{J\smallsetminus i} \hookrightarrow K_J$ for all $J\subseteq [m]$, the secondary operations $\delta_{ij}$ are essentially determined by the maps
\begin{equation} \label{two_maps}
K_{J\smallsetminus i}\cup K_{J\smallsetminus j} \hookrightarrow K_J \quad\text{ and }\quad K_{J\smallsetminus i}\cup K_{J\smallsetminus j} \hookrightarrow \Sigma K_{J\smallsetminus ij}.
\end{equation}
In particular, for each $J\subseteq [m]$ containing $i,j$, there is a cofibration sequence
\[
\begin{tikzcd}
K_{J\smallsetminus i}\vee K_{J\smallsetminus j} \ar[r] & K_{J\smallsetminus i}\cup K_{J\smallsetminus j} \ar[r] \ar[d] & \Sigma K_{J\smallsetminus ij} \ar[dl,"\ell",dashed] \\
 & K_J &
\end{tikzcd}
\]
inducing the Mayer--Vietoris sequence, and when the composite of the two leftmost arrows is null homotopic, there exists an extension $\ell$ inducing the lift in~\eqref{MV_diagram}.

\begin{example} \label{e_connecting_map}
Let $K$ be the simplicial complex on the vertex set $[5]$ with minimal non-faces $13$, $14$, $24$, $25$ and $345$. ($K$ is obtained from the boundary of a pentagon by adding the edge $35$.) Take $J=[5]$ and consider the Mayer--Vietoris sequence associated to the cover $\{K_{J\smallsetminus 1}, K_{J\smallsetminus 3}\}$. In this case, $K_{J\smallsetminus 1}\cup K_{J\smallsetminus 3} = K_J$, so the vertical map in~\eqref{MV_diagram} is an isomorphism. The map $K_{J\smallsetminus 1}\cup K_{J\smallsetminus 3} \hookrightarrow \Sigma K_{J\smallsetminus 13}$ inducing the connecting homomorphism (pictured in Figure~\ref{fig1}) is given up to homotopy by a map $S^1\vee S^1 \xrightarrow{1\vee\ast} S^1\vee \{pt\}$ collapsing the second wedge summand to a point. If $\alpha\in \widetilde{H}^1(S^1\vee S^1)\cong \widetilde{H}^1(K_J) \subset H^*(\mathcal{Z}_K)$ is a generator for the first circle summand, then $\alpha$ is in the kernel of the restriction maps $\widetilde{H}^1(K_J) \to \widetilde{H}^1(K_{J\smallsetminus 1})$ and $\widetilde{H}^1(K_J) \to \widetilde{H}^1(K_{J\smallsetminus 3})$, and hence $\iota_1\alpha=\iota_3\alpha=0$. Moreover, since $\alpha$ is clearly in the image of the connecting homomorphism, it follows that the multidegree $J$ part of $H^*(\mathcal{Z}_K)$ supports a nontrivial secondary operation $\delta_{13}\alpha\neq 0$.
\end{example}

\begin{figure}[ht]
    \centering
    \begin{tikzpicture}[font=\tiny, baseline=(current bounding box.center)]
        \coordinate [label=above:$1$] (1) at (1,1.5);
        \coordinate [label=left:$2$] (2) at (0,0);
        \coordinate [label=below:$3$] (3) at (1,-1.5);
        \coordinate [label=left:$4$] (4) at (1.2,0);
        \coordinate [label=right:$5$] (5) at (2,0);
        \coordinate [label = below: \footnotesize $K_{\{1,2,4,5\}}\cup K_{\{2,3,4,5\}}$] (a) at (1,-2);
        \draw (3) -- (4) -- (5);
        \draw (1) -- (2) -- (3) -- (5) -- (1);
        \foreach \point in {1,2,3,4,5}
            \fill [black] (\point) circle (1.5 pt);
        \draw[decorate,decoration={brace,amplitude=4pt,raise=10pt}] (0,0.05) -- (0,1.5) node[midway,xshift=-1.2cm] {$K_{\{1,2,4,5\}}$};
        \draw[decorate,decoration={brace,amplitude=4pt,raise=10pt}] (0,-1.5) -- (0,-0.05) node[midway,xshift=-1.2cm] {$K_{\{2,3,4,5\}}$};
    \end{tikzpicture}
$\quad \lhook\joinrel\longrightarrow \quad\;$
    \begin{tikzpicture}[font=\tiny, baseline=(current bounding box.center)]
       \coordinate [label=above:$1$] (1) at (1,1.5);
        \coordinate [label=left:$2$] (2) at (0,0);
        \coordinate [label=below:$3$] (3) at (1,-1.5);
        \coordinate [label=left:$4$] (4) at (1.2,0);
        \coordinate [label=right:$5$] (5) at (2,0);
        \coordinate [label = below: \footnotesize $\Sigma K_{\{2,4,5\}}$] (a) at (1,-2);
        \filldraw[fill=black!15] (1) -- (4) -- (5);
        \filldraw[fill=black!15] (3) -- (4) -- (5);
        \draw (1) -- (2) -- (3) -- (5) -- (1);
        \foreach \point in {1,2,3,4,5}
            \fill [black] (\point) circle (1.5 pt);
        \draw[decorate,decoration={brace,amplitude=4pt,raise=10pt}] (2,1.5) -- (2,0.05) node[midway,xshift=1.4cm] {$K_{\{2,4,5\}}{\ast} \{1\}$};
        \draw[decorate,decoration={brace,amplitude=4pt,raise=10pt}] (2,-0.05) -- (2,-1.5) node[midway,xshift=1.4cm] {$K_{\{2,4,5\}}{\ast} \{3\}$};
    \end{tikzpicture}
\caption{The Mayer--Vietoris connecting map $K_{J\smallsetminus i}\cup K_{J\smallsetminus j} \hookrightarrow \Sigma K_{J\smallsetminus ij}$\\ in \cref{e_connecting_map}} \label{fig1}
\end{figure}

\begin{remark}
In terms of the minimal free resolution $F$ of $k[K]$, the maps induced by the inclusions of subcomplexes~\eqref{two_maps} yield a partial combinatorial interpretation for the quadratic component of the differential. Together with the description of the linear component of the differential given by \cite[Theorem~1.1]{K} (or \cref{comm_diagrams}), this amounts to a combinatorial interpretation of the complex $F/\mathfrak{m}^3F$ that partially answers Katth\"an's Question~4.2 of \cite{K}.  One can remove the indeterminacy by fixing chain level data as in \cref{s_HB_pert}, but to fully answer Katth\"an's question one would need to understand how to make these choices at the cohomology level, in terms of the inclusions \eqref{two_maps} and Hochster's decomposition.
\end{remark}

We also obtain the following characterisation of equivariant formality, extending \cref{1-torus_ef} to the case of coordinate $2$-torus actions. Since this result will be further generalised in \cref{t_main_eq_formal}, we omit the proof. (See \cref{s_equivariant_formality_general} for the definition of the face deletion $K\smallsetminus F$.)

\begin{theorem}
Let $K$ be a simplicial complex on vertex set $[m]$ and let $I=\{i,j\}\subseteq [m]$ with $i\neq j$. Then the following conditions are equivalent:
\begin{enumerate}[label={\normalfont(\alph*)}]
\item the coordinate $T^I$-action on $\mathcal{Z}_K$ is equivariantly formal (over $k$);
\item $\delta_i$, $\delta_j$ and $\delta_{ij}$ are trivial on $H^\ast(\mathcal{Z}_K;k)$;
\item $K_J\smallsetminus(I\cap J) \hookrightarrow K_J$ induces the trivial map on $\widetilde{H}^\ast(\; ;k)$ for all $J\subseteq [m]$.
\end{enumerate}
\end{theorem}

\subsection{Higher operations and the Mayer--Vietoris spectral sequence}
\label{s_MV_ss_description}
We generalise the results of the previous section to the case of an arbitrary coordinate subtorus, that is, the $T^I$-action on $\mathcal{Z}_K$ for any $I\subseteq [m]$. In place of the Mayer--Vietoris long exact sequence used to describe secondary cohomology operations in the $|I|=2$ case, we will here identify all higher operations with differentials in a Mayer--Vietoris spectral sequence.

Fix $I\subseteq [m]$ and consider the $T^I$-action on $\mathcal{Z}_K$. To describe the higher cohomology operations induced by this torus action in terms of the Hochster decomposition $H^\ast(\mathcal{Z}_K)\cong \bigoplus_{J\subseteq [m]} \widetilde{H}^*(K_J)$, we fix a multidegree $J\subseteq [m]$ and, as before, consider the cover
\[
\mathcal{U}_{I\!,J} =\{ K_{J\smallsetminus i} ~:~ i\in I\cap J \}.
\]
The \emph{(ordered) \v{C}ech complex} $\check{C}^*(\mathcal{U}_{I\!,J},\widetilde{C}^p)$ of the cover $\mathcal{U}_{I\!,J}$ with coefficients in the presheaf $\widetilde{C}^p$ is given by
\[
\check{C}^q(\mathcal{U}_{I\!,J},\widetilde{C}^p) = \bigoplus_{\substack{i_0<\cdots<i_q \\ i_0,\ldots,i_q\in I\cap J}} \widetilde{C}^p(K_{J\smallsetminus i_0\cdots i_q}).
\]
For an element $\omega$ of $ \check{C}^q(\mathcal{U}_{I\!,J},\widetilde{C}^p)$, we write $\omega_{i_0\cdots i_q}$ for its component in $\widetilde{C}^p(K_{J\smallsetminus i_0\cdots i_q})$. The \v{C}ech differential is then defined by
\[
\check{d}\colon \check{C}^{q-1}(\mathcal{U}_{I\!,J},\widetilde{C}^p) \to \check{C}^q(\mathcal{U}_{I\!,J},\widetilde{C}^p),
\quad \quad
(\check{d}\omega)_{i_0\cdots i_q} = \sum_{\ell=0}^q(-1)^\ell \omega_{i_0\cdots\widehat{i_\ell}\cdots i_q}\big|_{K_{J\smallsetminus i_0\cdots i_q}}.
\]
We form the \v{C}ech double complex $\check{C}^q(\mathcal{U}_{I\!,J},\widetilde{C}^p)$, whose vertical differential is $(-1)^p\check{d}$, and whose horizontal differential $d\colon \check{C}^q(\mathcal{U}_{I\!,J},\widetilde{C}^p)\to \check{C}^q(\mathcal{U}_{I\!,J},\widetilde{C}^{p+1})$ is induced by the simplicial cochain differential. The inclusions $K_{J\smallsetminus i} \hookrightarrow K_J$ induce a morphism from $\widetilde{C}^*(K_J)$ to the  \v{C}ech double complex, and with this we form the augmented \v{C}ech double complex $a\check{C}^*(\mathcal{U}_{I\!,J},\widetilde{C}^*)$:
\begin{equation} \label{double_comp2}
\begin{tikzcd}
\vdots & \vdots & \vdots & \\
\bigoplus\limits_{i<j}\widetilde{C}^0(K_{J\smallsetminus ij}) \ar[u,"\check{d}"'] \ar[r,"d"] & \bigoplus\limits_{i<j}\widetilde{C}^1(K_{J\smallsetminus ij}) \ar[u,"-\check{d}"'] \ar[r,"d"] & \bigoplus\limits_{i<j}\widetilde{C}^2(K_{J\smallsetminus ij}) \ar[u,"\check{d}"'] \ar[r,"d"] & \cdots \\
\bigoplus\limits_{i}\widetilde{C}^0(K_{J\smallsetminus i}) \ar[u,"\check{d}"'] \ar[r,"d"] & \bigoplus\limits_{i}\widetilde{C}^1(K_{J\smallsetminus i}) \ar[u,"-\check{d}"'] \ar[r,"d"] & \bigoplus\limits_{i}\widetilde{C}^2(K_{J\smallsetminus i}) \ar[u,"\check{d}"'] \ar[r,"d"] & \cdots \\
\widetilde{C}^0(K_J) \ar[u] \ar[r,"d"] & \widetilde{C}^1(K_J) \ar[u] \ar[r,"d"] & \widetilde{C}^2(K_J) \ar[u] \ar[r,"d"] & \cdots
\end{tikzcd}
\end{equation}
Taking cohomology with respect to the horizontal differential yields the first page of the \emph{augmented Mayer--Vietoris spectral sequence} $(E^{p,q}_r,d_r)$ associated to $\mathcal{U}_{I\!,J}$, with
\begin{equation} \label{ss}
E_1^{p,-1}=\widetilde{H}^p(K_J), \qquad E_1^{p,q}=\bigoplus_{\substack{i_0<\cdots<i_q \\ i_0,\ldots,i_q\in I\cap J}} \widetilde{H}^p(K_{J\smallsetminus i_0\cdots i_q}) \;\text{ for } q\geqslant 0.
\end{equation}
(Equivalently, this is the spectral sequence associated to the filtration of the total complex of~\eqref{double_comp2} by column degree.)
The differential $d_1$, being induced by the vertical differential in~\eqref{double_comp2}, therefore has components given up to sign by the primary cohomology operations $\iota_i=\delta_i$ for $i\in I\cap J$. The next result identifies the higher differentials $d_s$ in this spectral sequence with the higher cohomology operations $\delta_U$ indexed by subsets $U\subseteq I\cap J$ with $|U|=s$.

\begin{lemma} \label{tech_lem}
Let $K$ be a simplicial complex on the vertex set $[m]$, and fix $I,J\subseteq [m]$. Suppose that $d_r=0$ for $1\leqslant r< s$ in the augmented Mayer--Vietoris spectral sequence associated to $\mathcal{U}_{I\!,J}$. Then the differential $d_s\colon E_s^{p,-1} \to E_s^{p-s+1,s-1}$ defines a map
\[
\widetilde{H}^p(K_J) \longrightarrow \bigoplus_{i_0<\cdots<i_{s-1}} \widetilde{H}^{p-s+1}(K_{J\smallsetminus i_0\cdots i_{s-1}})
\]
with each component given by $(-1)^{\varepsilon(i_0\ldots i_{s-1},J)+p+s}\delta_{i_0\cdots i_{s-1}}$.
\end{lemma}

\begin{proof}
Let $[\alpha] \in \widetilde{H}^p(K_J)$.
Observe that $d_1[\alpha]=\bigl([\alpha|_i]\bigr)_{i\in I\cap J} \in \bigoplus_i\widetilde{H}^p(K_{J\smallsetminus i})$, where for each $i\in I\cap J$ we have
\[ [\alpha|_i] = (-1)^{\varepsilon(i,J)+p+1} [\iota_i\alpha] = (-1)^{\varepsilon(i,J)+p+1} \delta_i[\alpha] \]
by Remark~\ref{mod_structures}. So $d_1[\alpha]=0$ implies that there exists an element $(\beta_i)_{i\in I\cap J} \in \bigoplus_i\widetilde{C}^{p-1}(J_{K\smallsetminus i})$ with $d(\beta_i)=\alpha|_i$ for each $i\in I\cap J$, and $d_2[\alpha]$ is then represented by $(-1)^{p-1}\check{d}\bigl((\beta_i)_{i\in I\cap J}\bigr)$, as indicated in the cochain-level zig-zag below.
\[
\begin{tikzcd}
\bigoplus\limits_{i<j} \widetilde{H}^{p-1}(K_{J\smallsetminus ij}) & \bigoplus\limits_{i<j} \widetilde{H}^p(K_{J\smallsetminus ij}) & (-1)^{p-1}\! \left(\beta_j|_i-\beta_i|_j\right)_{i<j} & \\
\bigoplus\limits_i \widetilde{H}^{p-1}(K_{J\smallsetminus i}) & \bigoplus\limits_i \widetilde{H}^p(K_{J\smallsetminus i}) & \left(\beta_i\right)_i \ar[u,mapsto,"(-1)^{p-1}\check{d}"'] \ar[r,mapsto,"d"] & \left(\alpha|_i\right)_i \\
\widetilde{H}^{p-1}(K_J) & \widetilde{H}^p(K_J) \ar[uul,"d_2"',near end] & & \alpha \ar[u,mapsto] 
\end{tikzcd}
\]
Since the cohomology class of $\left(\beta_j|_i-\beta_i|_j\right)_{i<j} \in \bigoplus_{i<j}\widetilde{C}^{p-1}(J_{K\smallsetminus ij})$ does not depend on the choice of $d$-preimages $\beta_i$ of $\alpha|_i$, we may assume for each $i\in I\cap J$ that $\beta_i=(-1)^{\varepsilon(i,J)+p+1}d^{-1}\iota_i\alpha$ for some preimage $d^{-1}\iota_i\alpha$ of $\iota_i\alpha$. Therefore, for each component of $\left(\beta_j|_i-\beta_i|_j\right)_{i<j}$, we have
\begin{align*}
\beta_j|_i-\beta_i|_j &= (-1)^{\varepsilon(i,J\smallsetminus j)+p}\iota_i\beta_j - (-1)^{\varepsilon(j,J\smallsetminus i)+p}\iota_j\beta_i \\
&= (-1)^{\varepsilon(i,J\smallsetminus j)+\varepsilon(j,J)+1}\iota_id^{-1}\iota_j\alpha - (-1)^{\varepsilon(j,J\smallsetminus i)+\varepsilon(i,J)+1}\iota_jd^{-1}\iota_i\alpha.
\end{align*}
Since $i<j$, it follows that $\varepsilon(i,J\smallsetminus j)=\varepsilon(i,J)$ and $\varepsilon(j,J\smallsetminus i)=\varepsilon(j,J)-1$, and hence
\begin{align*}
\left[\beta_j|_i-\beta_i|_j\right] &= (-1)^{\varepsilon(i,J)+\varepsilon(j,J)+1}\left[\iota_id^{-1}\iota_j\alpha+\iota_jd^{-1}\iota_i\alpha\right] \\
&= (-1)^{\varepsilon(ij,J)+1}\delta_{ij}[\alpha].
\end{align*}
Finally, it follows that each component of $d_2[\alpha]$ is of the form
\[ (-1)^{p-1}\left[\beta_j|_i-\beta_i|_j\right]=(-1)^{\varepsilon(ij,J)+p}\delta_{ij}[\alpha], \]
as claimed.

Proceeding inductively, we assume that $d_r=0$ for $1\leqslant r<s$ and that the components of 
\[
d_s[\alpha]\in E_s^{p-s+1,s-1} \cong \bigoplus_{i_0<\cdots<i_{s-1}} \widetilde{H}^{p-s+1}(K_{J\smallsetminus i_0\cdots i_{s-1}})
\]
are of the form $(-1)^{\varepsilon(i_0\ldots i_{s-1},J)+p+s} \delta_{i_0\cdots i_{s-1}}[\alpha]$. Now suppose $d_s$ is also trivial. To ease notation, we will write $\delta_{i_0\cdots i_{s-1}}(\alpha)$ for a cochain representative of $\delta_{i_0\cdots i_{s-1}}[\alpha]$ defined recursively by 
\[
\delta_{i_0\cdots i_{s-1}}(\alpha) = \sum_{\ell=0}^{s-1}\iota_{i_\ell}d^{-1}\delta_{i_0\cdots\widehat{i_\ell}\cdots i_{s-1}}(\alpha) \in \widetilde{C}^{p-s+1}(K_{J\smallsetminus i_0\cdots i_{s-1}}),
\]
where $d^{-1}\delta_{i_0\cdots\widehat{i_\ell}\cdots i_{s-1}}(\alpha)$ denotes a choice of preimage of $\delta_{i_0\cdots\widehat{i_\ell}\cdots i_{s-1}}(\alpha)$.
Then $d_s[\alpha]=0$ implies that there exists a zig-zag in the double complex~\eqref{double_comp2} of the form
\[
\begin{tikzcd}[column sep=0mm]
(-1)^{p-s}\left( \sum\limits_{\ell=0}^s(-1)^\ell\beta_{i_0\cdots\widehat{i_\ell}\cdots i_s}\big|_{i_\ell} \right)_{i_0<\cdots<i_s} & & & \\
\left(\beta_{i_0\cdots i_{s-1}}\right)_{i_0<\cdots<i_{s-1}} \ar[u,mapsto,"(-1)^{p-s}\check{d}"']\ar[r,mapsto,"d"] & \!\left((-1)^\varepsilon\delta_{i_0\cdots i_{s-1}}(\alpha)\right)_{i_0<\cdots<i_{s-1}} & & \\
 & \quad \ar[u,mapsto,"(-1)^{p-s+1}\check{d}"']\ar[rd, draw=none, "{\ddots}"] & & \\
 & & \!\!\ar[r,mapsto,"d"] & \left(\alpha|_{i_0}\right)_{i_0} \\
 & & & \hspace{4.4em}\alpha\in \widetilde{C}^p(K_J) \ar[u,mapsto]
\end{tikzcd}
\]
where $d(\beta_{i_0\cdots i_{s-1}})=(-1)^\varepsilon\delta_{i_0\cdots i_{s-1}}(\alpha)$ with $\varepsilon=\varepsilon(i_0\ldots i_{s-1},J)+p+s$ for each strictly increasing sequence $i_0,\ldots,i_{s-1}\in I\cap J$, and the top-left cochain is a representative of $d_{s+1}[\alpha]$. Now for each strictly increasing sequence $i_0,\ldots,i_s\in I\cap J$ and for each $0\leqslant\ell\leqslant s$, we have
\begin{align*}
\beta_{i_0\cdots\widehat{i_\ell}\cdots i_s}\big|_{i_\ell} &= (-1)^{\varepsilon(i_\ell,J\smallsetminus i_0\cdots\widehat{i_\ell}\cdots i_s)+p-s+1} \iota_{i_\ell}\beta_{i_0\cdots\widehat{i_\ell}\cdots i_s} \\
&= (-1)^{\varepsilon(i_\ell,J)+\ell+p-s+1} \iota_{i_\ell}\beta_{i_0\cdots\widehat{i_\ell}\cdots i_s}.
\end{align*}
It follows that each component of $d_{s+1}[\alpha]$ is represented by a cochain of the form
\begin{align*}
(-1)^{p-s}\sum_{\ell=0}^s(-1)^\ell \beta_{i_0\cdots\widehat{i_\ell}\cdots i_s}\big|_{i_\ell} 
&= \sum_{\ell=0}^s(-1)^{\varepsilon(i_\ell,J)+1} \iota_{i_\ell}\beta_{i_0\cdots\widehat{i_\ell}\cdots i_s} \\
&= \sum_{\ell=0}^s(-1)^{\varepsilon(i_\ell,J)+\varepsilon(i_0\ldots\widehat{i_\ell}\ldots i_s,J)+p+s+1} \iota_{i_\ell}d^{-1}\delta_{i_0\cdots\widehat{i_\ell}\cdots i_s}(\alpha) \\
&= (-1)^{\varepsilon(i_0\ldots i_s,J)+p+s+1}\delta_{i_0\cdots i_s}(\alpha),
\end{align*}
which closes the induction.
\end{proof}

Taking $J=[m]$ yields the augmented Mayer--Vietoris spectral sequence associated to the cover $\mathcal{U}_{I} =\{ K_{[m]\smallsetminus i} ~:~ i\in I\}$, and we will see next that this spectral sequence contains all of the higher operations as components in its differentials.

\begin{theorem}
\label{c_specseq}
Suppose that $d_r=0$ for $1\leqslant r< s$ in the augmented Mayer--Vietoris spectral sequence associated to $\mathcal{U}_{I}$. Then the differential $d_s\colon E_s^{p,q} \to E_s^{p-s+1,s+q}$ defines a map

\[
\bigoplus_{U\subseteq I,\ |U|=m-q-1} \widetilde{H}^p(K_{[m]\smallsetminus U}) \longrightarrow \bigoplus_{V\subseteq I,\ |V|=m-q-s-1} \widetilde{H}^{p-s+1}(K_{[m]\smallsetminus V}),
\]
and the components are given by $(-1)^{\varepsilon(V\smallsetminus U,[m])+p+s} \delta_{V\smallsetminus U}$ when $U\subseteq V$, and zero otherwise.
\end{theorem}

\begin{proof}
Writing $J=[m]\smallsetminus U$, there is a comparison map of augmented \v{C}ech double complexes
\[
a\check{C}^q(\mathcal{U}_{I\!,J},\widetilde{C}^p) \longrightarrow a\check{C}^{q+u}(\mathcal{U}_{I},\widetilde{C}^p),
\]
 increasing vertical degree by $u=|U|$. On the component indexed by $i_0\ldots i_q$ this is given by 
 \[
\widetilde{C}^p(K_{J\smallsetminus i_0\ldots i_q})\xrightarrow{(-1)^{\varepsilon(i_0\ldots i_{q},U)}}\widetilde{C}^p(K_{[m]\smallsetminus i_0\ldots i_q,U}),
 \]
 using the equality $K_{J\smallsetminus i_0\ldots i_q}=K_{[m]\smallsetminus i_0\ldots i_q,U}$ to identify the two sides. On the first page of the associated spectral sequences this map restricts to the isomorphism 
  \[
\widetilde{H}^p(K_{J\smallsetminus i_0\ldots i_q})\xrightarrow{(-1)^{\varepsilon(i_0\ldots i_{q},U)}}\widetilde{H}^p(K_{[m]\smallsetminus i_0\ldots i_q,U}).
 \]
 Under the assumption that $d_r=0$ for $1\leqslant r< s$ in the spectral sequence associated to $\mathcal{U}_{I}$, we may also assume by induction that $d_r=0$ for $1\leqslant r< s$ in the spectral sequence associated to $\mathcal{U}_{I\!,J}$. Therefore on the $s$th page the comparison map induces commutative squares
   \[
   \begin{tikzcd}[column sep=30mm]
       \bigoplus \widetilde{H}^{p-s+1}(K_{J\smallsetminus i_0\ldots i_{s+q}})\ar[r,"\oplus{(-1)^{\varepsilon(i_0\ldots i_{s+q},U)}}"] & \bigoplus \widetilde{H}^{p-s+1}(K_{[m]\smallsetminus i_0\ldots i_{s+q},U})\\
       \widetilde{H}^p(K_{J\smallsetminus i_0\ldots i_q}) \ar[r,"(-1)^{\varepsilon(i_0\ldots i_{q},U)}"]\ar[u,"d_s"] & \widetilde{H}^p(K_{[m]\smallsetminus i_0\ldots i_q,U}).\ar[u,"d_s"]
   \end{tikzcd}
\]
In particular, beginning on the $-1$st row and taking $V=\{i_0,\ldots,i_{s-1}\}\cup U$, the claimed formula follows from \cref{tech_lem}, with the sign $(-1)^{\varepsilon(V\smallsetminus U,U)+\varepsilon(V\smallsetminus U,J)+p+s}=(-1)^{\varepsilon(V\smallsetminus U,[m])+p+s}$.
\end{proof}

\subsection{Equivariant formality from combinatorics}\label{s_equivariant_formality_general} 

We are ready to prove our main result characterising the equivariant formality of subtorus actions on moment-angle complexes $\mathcal{Z}_K$, purely in terms of the cohomology of subcomplexes of $K$. By \cref{p_coord_torus}, it suffices for us to treat the case of coordinate subtori.

The \emph{face deletion} of a simplicial complex $K$ at $F\in K$ is the largest subcomplex $K\smallsetminus F$ of $K$ that does not contain $F$:
\[
K\smallsetminus F=\{\sigma\in K \;:\; F\not\subseteq\sigma\}.
\]
It follows from the definition that the face deletion can be written as the union of full subcomplexes
\[
K\smallsetminus F=\bigcup_{i\in F}K_{[m]\smallsetminus i}.
\]
By convention, $K\smallsetminus \varnothing$ is the empty simplicial complex. We also remind the reader of the notation
\[
\mathcal{U}_{I\!,J} =\{ K_{J\smallsetminus i} ~:~ i\in I\cap J \}
\]
and that this collection of subcomplexes of $K_J$ induces a Mayer--Vietoris spectral sequence as in \cref{s_MV_ss_description}. The collection $\mathcal{U}_{I}=\mathcal{U}_{I\!,[m]}$ is of particular importance.

\begin{theorem}\label{t_main_eq_formal}
Let $K$ be a simplicial complex on vertex set $[m]$ and let $I\subseteq [m]$. Then the following conditions are equivalent:
\begin{enumerate}[label={\normalfont(\alph*)}]
\item\label{i_eqform} the coordinate $T^I$-action on $\mathcal{Z}_K$ is equivariantly formal over $k$;
\item\label{i_opszero} the cohomology operations $\delta_J$ vanish on $H^*(\mathcal{Z}_K;k)$ for all $J\subseteq I$;
\item\label{i_ss} the augmented Mayer--Vietoris spectral sequence associated to $\mathcal{U}_{I}$ degenerates at its first page (or equivalently, $\mathcal{U}_{I\!,J}$ for all $J$); 
\item\label{i_comb} $K_J\smallsetminus(I\cap J) \hookrightarrow K_J$ induces the trivial map on $\widetilde{H}^\ast(\; ;k)$ for all $J\subseteq [m]$.
\end{enumerate}
\end{theorem}

\begin{proof}
    The equivalence of \ref{i_eqform} and \ref{i_opszero} is in \cref{c_eq_form_and_coh_ops}. The extra equivalence smuggled into \ref{i_ss} follows from the argument given for \cref{c_specseq}, since, for any $J$, the differentials appearing augmented Mayer--Vietoris spectral sequence associated to $\mathcal{U}_{I\!,J}$ appear in that of $\mathcal{U}_{I}$. After this, \ref{i_opszero} is equivalent to \ref{i_ss} by \cref{c_specseq}. So it is sufficient to show that \ref{i_opszero} is equivalent to \ref{i_comb}.

    For brevity, we fix $J$ and write $\check{C}_J$ for the \v{C}ech double complex associated to $\mathcal{U}_{I\!,J}$, and $a\check{C}_J$ for the corresponding augmented \v{C}ech double complex \eqref{double_comp2}. There is then a short exact sequence of complexes
    \begin{equation}\label{e_ses_tot}
    0\longrightarrow\operatorname{Tot}(\check{C}_J) \longrightarrow \operatorname{Tot}(a\check{C}_J) \longrightarrow \widetilde{C}^*(K_J)[-1]\longrightarrow 0.
    \end{equation}
    It is well known that the \v{Cech} complex  associated $\mathcal{U}_{I\!,J}$ is acyclic, that is, for each $p$ there is a quasi-isomorphism $\widetilde{C}^p(\bigcup_{i\in I\cap J}K_{J\smallsetminus i})\xrightarrow{\simeq}\check{C}^*(\mathcal{U}_{I\!,J},\widetilde{C}^p)$. It follows that there is a quasi-isomorphism 
    $\widetilde{C}^*(\bigcup_{i\in I\cap J}K_{J\smallsetminus i})\xrightarrow{\simeq}\operatorname{Tot}(\check{C}_J)$. 
    We also note that $\bigcup_{i\in I\cap J}K_{J\smallsetminus i}= K_J\smallsetminus(I\cap J)$ is the face deletion. 
    Therefore, by taking cohomology \eqref{e_ses_tot} yields a long-exact sequence
    \[
    \cdots\longrightarrow \widetilde{H}^*(K_J\smallsetminus(I\cap J))\longrightarrow H^*(\operatorname{Tot}(a\check{C}_J)) \xrightarrow{\ e\ } \widetilde{H}^{*-1}(K_J)\xrightarrow{\ \delta\ } \widetilde{H}^{*-1}(K_J\smallsetminus(I\cap J))\longrightarrow \cdots
    \]
     By exactness, $e$ is surjective if and only if $\delta$ is zero; moreover the connecting homomorphism $\delta$ is induced by the inclusion $K_J\smallsetminus(I\cap J) \hookrightarrow K_J$, so this is equivalent to \ref{i_comb}. The edge map $e$ comes from the projection of $a\check{C}_J$ onto its $-1$st row. Therefore $e$ is surjective if and only if, in the corresponding spectral sequence \eqref{ss}, every differential $d_s$ leaving the $-1$st row is zero, for $s\geqslant 1$. By \cref{tech_lem} this is equivalent to \ref{i_opszero}. 
\end{proof}

\end{document}